\documentclass[12pt]{article}

\usepackage{epsfig}
\usepackage[margin=1.15in]{geometry}
 \usepackage{graphicx}
\DeclareGraphicsExtensions{.pdf,.eps}

\usepackage{amsfonts, amssymb, amsmath, amsthm}
\usepackage[active]{srcltx}
\usepackage[nodisplayskipstretch]{setspace}
\usepackage[latin1]{inputenc}

\newcommand{\cR}{\mathcal{R}}

\newcommand{\RR}{\mathbb{R}}

\newcommand{\nee}{\mathcal{NE}}

\def\cE{\mathcal{E}}
\def\cS{\mathcal{S}}
\def\noi{\noindent}

\newcommand{\De}{\Delta}

\renewcommand{\phi}{\varphi}
\renewcommand{\leq}{\leqslant}
\renewcommand{\geq}{\geqslant}

\newcommand{\la}{\lambda}

\newcommand{\ga}{\gamma}

\newcommand{\de}{\delta}
\newcommand{\e}{\epsilon}

\newtheorem{obs}{Remark}

\newtheorem{lem}{Lemma}
\newtheorem{cor}{Corollary}
\newtheorem{teo}{Theorem}
\newtheorem{defi}{Definition}

\begin{document}

\title{The Evolutionary Robustness of Forgiveness and Cooperation}

\author{Pedro Dal B\'o\thanks{Department of Economics, Brown University and NBER.} \,\,\,\,\,\,\,\,\,Enrique R. Pujals\thanks{Graduate Center, CUNY and Instituto de Matem\'atica Pura e Aplicada. Pujals gratefully acknowledges the support of NSF via grant DMS-1956022. Contents are solely the responsibility of the authors and do not necessarily represent the official views of the NSF.}}
\doublespacing

 \maketitle

\begin{abstract} We study the evolutionary robustness of strategies in infinitely repeated prisoners' dilemma games in which players make mistakes with a small probability and are patient. The evolutionary process we consider is given by the replicator dynamics. We show that there is a large class of strategies with a uniformly large basin of attraction independent of the finite set of strategies involved. Moreover, we show that those strategies can not be unforgiving and, assuming that they are symmetric, they cooperate. We provide partial efficiency results for asymmetric strategies.\end{abstract}

\section{Introduction}

The theory of infinitely repeated games has been very influential in the social sciences showing how repeated interaction can provide agents with incentives to overcome opportunistic behavior. However, a usual criticism of this theory is that there may be a multiplicity of equilibria. While cooperation can be supported in equilibrium when agents are sufficiently patient, there are also equilibria with no cooperation. Moreover, a variety of different punishments can be used to support cooperation.

To solve this multiplicity problem, we study what types of strategies will have a large basin of attraction regardless of what other strategies are considered in the evolutionary dynamic. More precisely, we study the replicator dynamic over arbitrary finite set of strategies when strategies make mistakes with a small probability in every round of the game. 
 We study which strategies have a non-vanishing basin of attraction with a uniform size regardless of the set of strategies being considered. We say that a strategy has a uniformly large basin of attraction if it repels invasions  of a given size for arbitrarily patient players and small probability of errors and for any possible finite combination of alternative strategies.

We find that two well known strategies, ``Always Defect'' and ``Grim,'' do not have uniformly large basins of attraction.%
\footnote{The strategy Grim starts by cooperating and continues to do so unless there is a defection, in that case it defects for ever.}
Moreover, any strategy that is unforgiving cannot have a uniformly large basin either (we say that a strategy is unforgiving if there is a finite history after which the strategy always defects). The reason is that, as players become arbitrarily patient and the probability of errors becomes small, unforgiving strategies lose in payoffs relative to strategies that forgive and the size of the basins of attraction between these two strategies will favor the forgiving one. This is the case even when the inefficiencies happen in histories that are not reached without trembles (as it is the case for Grim).

We also show that symmetric strategies (their actions depends in what happened and not on who did it) leading to inefficient payoffs cannot have uniformly large basins of attraction.  We also provide some efficiency results for asymmetric strategies. 

It could be the case that inefficient and unforgiving strategies do not have uniformly large basins since there may be no strategies with that property! We prove that that is not the case by showing that the there is a large class of strategies, that we call star-type strategies, which have a uniformly large basin of attraction. The star-type class includes the well known strategy ``Win-stay-lose-shift'' (WSLS), and the family of ``Trigger'' strategies.%
\footnote{The strategy Win-stay-lose-shift starts by cooperating and then cooperates if the actions of the two players coincided in the previous period and defects otherwise. A trigger strategy starts by cooperating and then cooperates if there was no deviations in a given number of previous periods.}
As strategies in this class are efficient, we show that the concept of uniformly large basins of attraction provides a (partial) solution to the long studied problem of equilibrium selection in infinitely repeated games: only efficient equilibria survive for patient players if we focus on symmetric strategies.

Moreover, in our study of the replicator dynamics, we develop tools that can be used outside of the particular case of infinitely repeated games. In fact, the existence results are based on theorem about replicator dynamics which can be used to study the robustness of steady states for games in general.%
\footnote{While we focus on the replicator dynamics, all the results are also valid for any evolutionary dynamic in which strategies earning more than the average grow, and those earning less shrink.} For example, we show that to check that an attractor has a uniformly large basin of attraction against invasion by any finite set of invaders, it is enough to check a condition on payoffs that only considers invasion by pairs (the largest dimension of the simplex that needs to be considered is 3). This greatly simplifies the analysis.

An extensive previous literature has addressed the multiplicity problem in infinitely repeated games. Part of this literature focuses on strategies of finite complexity with costs of complexity to select a subset of equilibria (see Rubinstein 1986, Abreu and Rubinstein 1988, Binmore and Samuelson 1992, Cooper 1996, and Volij 2002). This literature finds that the selection varies with the equilibrium concept being used and the type of cost of complexity. Another literature appealed to ideas of evolutionary stability as a way to select equilibria and found that no strategy is evolutionary stable in the infinitely repeated prisoners' dilemma (Boyd and Lorberbaum 1987). The reason is that for any strategy there exists another strategy that differs only after events that are not reached by this pair of strategies. As such, the payoff from both strategies is equal when playing with each other, and the original strategy cannot be an attractor of an evolutionary dynamic. Bendor and Swistak (1997) circumvent the problem of ties by weakening the stability concept, and show that cooperative and retaliatory strategies are the most robust to invasions. 
Garc\'ia and van Veelen (2016) use an alternative weakening of the stability concept and find that no equilibrium is robust if the players are sufficiently patient. 

In a different approach to ties, Boyd (1989) introduced the idea of errors in decision making. If there is a small probability of errors in every round, then all events in a game occur with positive probability destroying the certainty of ties allowing for some strategies to be evolutionary stable. However, as shown by Boyd (1989) and Kim (1994), many strategies that are subgame perfect for a given level of patience and errors can also be evolutionary stable.

Fudenberg and Maskin (1993) (see also Fudenberg and Maskin 1990) show that evolutionary stability can have equilibrium selection implications if we ask that the size of invasions that the strategy can repel to be uniformly large with respect to any
alternative strategy and for large discount factors and small probabilities of mistakes. They show that the only strategies with this characteristic must be cooperative. There are three main differences with our results. First, Fudenberg and Maskin (1993) focus on strategies of finite complexity while we do not have that restriction. Second, our robustness concept does not only consider the robustness to invasion by a single alternative strategy but also robustness to invasion by any arbitrary finite combination of alternative strategies. In other words, we also look at the size of the basin of attraction inside the simplex. Third, our full efficiency result only applies to the case of symmetric strategies and we only provide partial efficiency results for the general case.

Our results also relate to Johnson, Levine and Pesendorfer (2001), Volij (2002) and Levine and Pesendorfer (2007) who use stochastic stability (Kandori, Mailath and Rob 1993 and Young 1993) to select equilibria in infinitely repeated games. As having large basin of attraction is a necessary condition (but not sufficient) for stochastic stability, the present results could help characterize strategies that are stochastically stable for any finite set of strategies. 

\section{Model and definitions}
\label{sec: model}

We consider a homogeneous population of mass one playing, in each instant $\theta$ in the continuum time of evolution, the infinitely repeated prisoners' dilemma game we define below. Each agent plays one strategy in the infinitely repeated game among a predetermined finite set of possible strategies. The prevalence of each strategy in the population will evolve as function of the payoffs reached by each strategy. In particular we will assume that the dynamics of evolution are given by the replicator dynamic also defined below.

\subsection{Infinitely repeated prisoners' dilemma with trembles}
\label{trembles}

In each instant $\theta$ of the evolutionary time, agents are matched in pairs with each of the other agents to play the following infinitely repeated game. In each period of the infinitely repeated game $t=0,1,2,...$ the two agents play a symmetric stage game with action space $A=\left\{ C,D\right\} $. At each period $t$, a player chooses action $(\hat a^t)\in A$ and the other player chooses action $(\hat b^t)\in A$. However, the chosen action is only implemented with probability $p<1$, and with probability $1-p$ the other action is implemented: the players tremble. Trembles are independent across periods and players.
We denote the vector of own implemented actions until time $t$ as $a_t=(a^0, a^1,\dots, a^t)$ and $b_t=(b^0, b^1,\dots, b^t)$ for the other player. 
The payoff from the stage game at time $t$ is given by utility function $u(a^t, b^t):A\times A\rightarrow \Re $ such that $u(D,C)=T$, $u(C,C)=R$, $u(D,D)=P$,\ $u(C,D)=S$, with $T>R>P> S$ and $2R>T+S$.

Agents only observe previously implemented actions. This knowledge is summarized by public histories.
When the game begins we have the null history $h^{0}$, afterwards $
h_{t}=(a_{t-1}, b_{t-1})=((a^{0}, b^0),\dots (a^{t-1}, b^{t-1}))$ and $H_{t}$ is the space of all
possible $t$ histories. Let $H_{\infty}$ be the set of all possible
infinite histories and $H=\cup_{t\geq 0}H_t$ be the set of all possible finite histories.
A pure public strategy is a function $s: H\rightarrow A$.

In this paper we will have to pay attention to the mirror image of any finite history $h_t=(a_{t-1}, b_{t-1})$, which we define as $ \hat h_t:=(b_{t-1}, a_{t-1}).$ We will also pay attention to the history that strategies would generate if players never tremble: given a pair of strategies $(s,s')$ we denote the history that they generate without trembles as $h_{s,s'}$ and with ${h_{s,s'}}_t$ the finite history up to period $t-1$. Given a finite history $h_t$, with $h_{s,s'/h_t}$ we denote the history that $s$ and $s'$ generate from $h_t$ if players never tremble from then on. We call these histories 0-tremble histories. Note that 0-trembles denote zero future trembles; of courses, trembles may have been necessary to reach the particular history from which the 0-tremble history in consideration originates. Denote by $h^{\tau}_{s,s'/h_t}$ the 0-tremble actions at time $\tau>t$ if players follow strategies $s$ and $s'$ and history $h_t$ has been reached. 

The expected discounted payoff of a player following strategy $s$ while matched with a player following strategy $s'$ is $U_{\de,p}(s,s')= (1-\de)\sum_{t\geq 0} \de^t   p_{s,s'}(a_t,b_t)u(a^t,b^t)$ where $p_{s,s'}(a_t,b_t)$ denotes the probability that the history $(a_t,b_t)$ is reached when $s$ and $s'$ are the strategies used by the agent and the agent she is matched with respectively.
Observe that $ p_{s,s'}(a_t,b_t)=   p_{s,s'}(a_{t-1},b_{t-1}) p^{i_t+j_t} (1-p)^{1-i_t+1-j_t}$
 where   $i_t=1$ if $a^t = s(h_{t})$,  $i_t=0$ otherwise, and  $j_t=1$ if  $b^t = s'(\hat h_{t})$, $j_t=0$ otherwise.
Therefore, $ p_{s,s'}(a_t,b_t)=   p^{m_t+n_t} (1-p)^{2t+2-m_t-n_t}$
where $m_t= \#\{ 0\leq \tau \leq t: s_1(h_\tau)=a^\tau \}$ 
and $n_t= \#\{ 0\leq \tau \leq t: s_2(\hat h_\tau)=b^\tau \}.$
Observe that $p_{s,s'}(h_t)=p^{2t}$ if $h_t$ is a 0-tremble history.
By $U_{\de,p}(s,s' / h_t)$ we denote the expected discounted payoff conditional on history $h_t$ having been reached when $s$ and $s'$ are the strategies used by the agent and the agent she is matched with respectively: $U_{\de,p}(s,s' / h_t) = (1-\de)\sum_{\tau \geq t} \de^{\tau-t}   p_{s,s' / h_t}(a_\tau,b_\tau)u(a^\tau,b^\tau)$ where $p_{s,s' / h_t}(a_\tau,b_\tau)$ is the probability of history $(a_\tau,b_\tau)$ conditional on history $h_t$ having been reached (with $\tau > t$).

\subsection{Replicator dynamics}
\label{sec: replicator}

Since we are interested in studying the evolutionary stability of repeated game strategies, in this section we define the replicator dynamics for the case in which the matrix of payoffs is given by an infinitely repeated prisoners' dilemma game with discount factor $\delta$ and error probability $1-p$ for a finite set of strategies  ${\cal S}=\{s_1,\dots, s_n\}$. See Fudenberg and Levine (1998) and Weibull (1995) for further discussions on replicator dynamics and attractors in general.
  
Given a finite set of strategies ${\cal S}=\{s_1,\dots, s_n\}$, discount factor $\delta$ and error probability $1-p$, we can calculate the expected payoffs from any pair of strategies as discussed in the previous section. Let $U=(u_{ij})_{1\leq i\leq n, 1\leq j\leq n}$ be the square matrix with the expected payoffs for each pair of strategies - we drop $\delta$ and $1-p$ from the notation of $U$ for simplicity. Let $\De$ be the $n-$dimensional simplex  $\Delta=\{x=(x_1\dots x_n)\in \RR^n: x_1+\dots+x_n=1,\,\, x_i\geq 0,\forall i\}$ denoting all the possible distributions of the strategies in the population of agents. 
We consider the replicator dynamics associated to the payoff matrix $U$ on the $n$ dimensional simplex given by the equations $ \dot x_i=x_i[u_i(x)-\bar u (x)]$ where $u_i(x) = \sum_{j=1}^n x_ju_{ij}$ is the expected payoff of strategy $i$, and $\,\,\,\bar u(x)=\sum_{j=1}^n x_ju_j(x)$ is the average expected payoff. 
We denote with $\phi$ the associated flow that provides the solution of the replicator equation: $\phi:\RR\times \De\to \De.$  Observe that any vertex is a singularity of the replicator equation, therefore, any vertex is a fixed point of the flow.
Given a vertex $e_i$ (a $n-$dimensional vector with value $1$ in the $i-$th coordinate and zero in the others) and  $\e>0$, the set $\De_\e(e_i)=\{x: \sum_{j=1, j\neq i}^n x_j <\e\}$  denotes the ball of radius $\e$ and center $e_i.$ 

\begin{defi}{\bf Attracting fixed point and local basin of attraction}. A given vertex $e_i$ is an attractor if there exists an open neighborhood $V$ of $e_i$ such that for any $x\in V$ it follows that $\phi_\theta(x)\to e_i$ as $\theta \to +\infty$. The global basin of attraction of $e_i$, $B^s(e_i)$,  is the set of points with forward trajectories converging to $e_i$.
The local basin of attraction of $e_i$, $B_{loc}^s(e_i),$ is the set of points in $B^s(e_i)$ such that $\frac{\partial |\phi_\theta(x)-e_i|}{\partial \theta}<0$ (the flow gets closer to $e_i$).
\end{defi}

With $B_{loc}(s, \de, p, \cS)$ we denote the local basin of attraction of $s$ in the finite set of strategies ${\cal S}$. 

\subsection{Uniformly large basin of attraction in infinitely repeated  games} \label{replicator IPD}
                                                                                                                                
Since players tremble and the intended action is not public information, we focus on perfect public equilibria (see Fudenberg and Levine 1994):

\begin{defi} \label{sgp}  We say that $(s,s)$ is a strict perfect public equilibrium if there exists $p_0$ and $\de_0$ such that $U_{\de, p}( s,s/h_t)-U_{\de, p}(s',s/h_t)>0$ for any $h_t$, $s'$, $p>p_0$ and $\de>\de_0$.   
\end{defi}

It is clear that if $(s,s)$ is not a strict perfect public equilibrium then, $s$ will not repel invasions by every other strategy. It is also well known that a strict equilibrium strategy is an attractor in any population containing it.
However, the size of the basin of attraction of such a strategy could, in principle, be made arbitrarily small by appropriately choosing the set of alternative strategies, $\de$ and $p$. 
If a strategy can be invaded by one alternative strategy at a time, we could ask that the strategy resists invasions of a given size for any single invading strategy, as done in Fudenberg and Maskin (1990).

\begin{defi}\label{ULBS2} We say that a strategy $s$ has a uniformly large basin of attraction against single invaders if there exist numbers $K$, and $0<\de_0<1$, and an increasing function $p:[\de_0, 1]\to [0,1]$ (with $p(1)=1$) verifying that for any set of two strategies ${\cal S}$ containing $s$, identifying $s$ with the vertex $e_1,$ and any $\de>\de_0$ and $p>p(\de)$, it holds that $\{(x_1, x_2):\,  x_2\leq K\}\subset B_{loc}(s, p, \de, {\cal S}) $.
\end{defi}

The role of the function $p(\delta)$ is to bound the importance of mistakes relative to the patience parameter $\delta$. To avoid notation, we say, in short, that we restrict ourselves to $\de$ and $p$ large.

With only one invading strategy, the size of the basin of attraction of the original strategy depends on the payoff matrix of the $2\times 2$ game formed by these two strategies. If $(s,s)$ is a strict public perfect equilibrium, the size of the basin of attraction of strategy $s$ against strategy $s'$ is $f_{ss'}=\frac{1}{1+\frac{U_{\de,p}(s',s')-U_{\de,p}(s,s')}{U_{\de,p}(s,s)-U_{\de,p}(s',s)}},$ when $U_{\de,p}(s',s')>U_{\de,p}(s,s')$, and $f_{ss'}=1$ when $U_{\de,p}(s',s') \leq U_{\de,p}(s,s')$. Note that the size of the invasion by strategy $s'$ that strategy $s$ can resist is decreasing in the cost of miscoordinating when the other plays strategy $s'$, and increasing in the cost of miscoordinating when the other plays strategy $s$.

Therefore, a strategy $s$ has a uniformly large basin against single invaders if there exists a constant number $L$ such that for any $s'\neq s$ it follows that 
\begin{eqnarray}\label{eq 2 a 2}
\frac{U_{\de,p}(s', s')-U_{\de,p}(s, s')}{U_{\de,p}(s, s)-U_{\de,p}(s',s)}< L,
\end{eqnarray} 
for $\delta$ and $p$ large.%
\footnote{If we focus on the case without trembles, $p=1$, then the concept of uniformly large basin of attraction against single invaders coincides with the concept of uniform invasion barriers -- see Weibull (1995).}

However, strategies may not invade one at a time, and to study the robustness of a strategy we need to consider invasions by any combination of alternative strategies. In other words, we need to consider the size of the basin of attraction also inside the simplex, not only on its boundaries. Moreover, note that there is an infinite and uncountable number of alternative strategies in infinitely repeated games. This results in a large number of possible sets of invading strategies. To capture the idea of evolutionary robustness, in this paper we ask that a strategy has a large basin of attraction independently of the finite set of other participating strategies. That is, we ask that the strategy repels invasions of a given size for any possible set of invading strategies.
 
\begin{defi}\label{ULBS} We say that a strategy $s$ has a  uniformly large basin if there exist numbers $K$, and $0<\de_0<1$, and an increasing function $p:[\de_0, 1]\to [0,1]$ (with $p(1)=1$)  verifying that for any finite set of strategies ${\cal S}$ containing $s$, identifying $s$ with the vertex $e_1,$ and any $\de>\de_0$ and $p>p(\de)$, it holds that $\{(x_1, \dots, x_n):\,  x_2+\dots +x_n\leq K\}\subset B_{loc}(s, p, \de, {\cal S}) $ where $n=\#(\cS)$.
\end{defi}

It is important to realize that the robustness to invasion by combination of strategies may be quite different to robustness to invasion by single strategies. While it is clear that if $s$ has a uniformly large basin then it also has a uniformly large basin against single invaders, the converse may not be true. It could well be that the border of the basin of attraction bends inward in the interior of the simplex, depending on the payoffs that the invading strategies earn while interacting with each other. As such, calculating the size of the basin of attraction of a strategy would involve considering all the possible combinations of invading strategies, what could be quite consuming if the number of possible strategies is large. We show in Section \ref{large basin} that a simpler approach is possible under replicator dynamics. But first, in the next two sections, we provide a series of results that do not require that we travel to the interior of the simplex.

\section{The cost of unforgiveness}
\label{sec: unforgiveness}
One of the goals of this paper is to characterize the strategies that are evolutionary robust as captured by having uniformly large basins of attraction. We start by showing that unforgiving strategies are not evolutionary robust.

\begin{defi} We say that a strategy $s$ is unforgiving if there exists a history $h_t$ such that $s(h_{t}h_{\tau})=D$ for any history $h_{t}h_{\tau}$ that comes after history $h_t$.
\end{defi}

A commonly discussed unforgiving strategy is Always Defect ($a$). In games without trembles, Myerson (1991) proved that the basin of attraction of Always Defect collapses as the discount factor converges to one. The reason is that when Always Defect meets a cooperative strategy like Grim, the relative cost of miscoordination favors the cooperative strategy as the agents become more patient. Grim risks a low payoff in one period to secure high payoffs in all future periods if matched with another player playing Grim, and the relative value of that one period cost decreases as agents become more patient. Without trembles the size of the basin of attraction of Always Defect (a) against Grim (g) is $ f_{ag}=\frac{1}{1+\frac{U_{\de,p=1}(g,g)-U_{\de,p=1}(a,g)}{U_{\de,p=1}(a,a)-U_{\de,p=1}(g,a)}}=\frac{1}{1+\frac{R-(1-\delta)T-\delta P}{P-(1-\delta)S-\delta P}}$ which converges to zero as $\delta$ goes to one. The next lemma shows that this is also the case with small trembles.

\begin{lem}\label{lem: AD no se la banca} Always Defect ($a$) does not have a uniformly large basin of attraction. \end{lem}

\begin{proof} To prove that Always Defect does not have a uniformly large basin of attraction, we need to show that its basin of attraction can be made arbitrarily small for some large $\delta$ and $p$. Consider the invading strategy Grim, then the size of the basin of attraction of Always Defects is $\frac{1}{1+\frac{U_{\de,p}(g,g)-U_{\de,p}(a,g)}{U_{\de,p}(a,a)-U_{\de,p}(g,a)}}$. That is, Always Defect does not have a uniformly large basin of attraction if the miscoordination cost ratio, $\frac{U_{\de,p}(g,g)-U_{\de,p}(a,g)}{U_{\de,p}(a,a)-U_{\de,p}(g,a)}$, can be made arbitrarily large for large $\delta$ and $p$. Note that the miscoordination cost ratio converges to $\frac{U_{\de,p=1}(g,g)-U_{\de,p=1}(a,g)}{U_{\de,p=1}(a,a)-U_{\de,p=1}(g,a)}=\frac{R-(1-\delta)T-\delta P}{P-(1-\delta)S-\delta P}$ as $p$ goes to one. As in Myerson (1991), this ratio can be made arbitrarily large by choosing $\delta$ close to one. Hence, the basin of attraction of Always Defect can be made arbitrarily small.
\end{proof}

While it may not be surprising that Always Defect is not evolutionary robust, given that it obtains low payoffs as it defects from the beginning, we show next that the same applies to any unforgiving strategy - even when the history that triggers defection for ever may only be reached through trembles.

\begin{teo}\label{teo: unforgiving} Unforgiving strategies do not have a uniformly large basin of attraction. 
\end{teo}

As it will be clear from the proof, the reason that an unforgiving strategies does not have a uniformly large basin of attraction is that there exist an alternative strategy which would be willing to forgive and hence reach higher payoffs when playing with itself. To be more precise, if the unforgiving strategy $s$ defects for ever starting in history $h_t$, there is a strategy $s'$ that only differs from $s$ in that, starting in histories $h_t$ and its mirror history $\hat h_t$, $s'$ will cooperate for ever unless there is a defection. The cooperation by $s'$ at $h_t$ and $\hat h_t$  works as a ``secret handshake" (Robson 1990) and leads to persistent cooperation with itself and higher payoffs that those reached by the unforgiving strategy. This makes the basin of attraction of $s$ arbitrarily small.%
\footnote{Note that while the ``secret handshake" in Robson (1990) is costless, in our case it results in a one-period cost which becomes arbitrarily small as $\de$ goes to one.}

\begin{proof}
To prove that an unforgiving strategy $s$ does not have a uniformly large basin of attraction, we find an alternative strategy such that the basin of attraction of $s$ is arbitrary small against this alternative strategy for large $\de$ and $p$.

Since $s$ is unforgiving, there exists a history $h_t$ such that $s(h_{t}h_{\tau})=D$ for any $h_\tau.$ We consider first the case in which $h_t$ is a symmetric history ($h_t=\hat{h}_t$). Consider an invader strategy, $s'$ which behaves like $s$ in every history but history $h_t$ and those following it. In those histories, $s'$ behaves as a Grim strategy that disregards what happened before $h_t$ (that is, $s(h_t)=C$ and will cooperate in every history $h_th_\tau$ if, and only if, $h_\tau$ does not include a defection).

To prove that $s$ does not have a uniformly large basin of attraction, we must show that the ratio of misscoordination costs $\frac{U_{\de,p}(s', s')-U_{\de,p}(s, s')}{U_{\de,p}(s, s)-U_{\de,p}(s',s)}$ can be made arbitrarily large for $\de$ and $p$ close to one. Note that, given the definition of $s'$, the two strategies only differ on $h_t$ and following histories. Hence, $U_{\de,p}(s', s')-U_{\de,p}(s, s')= p_{s,s}( h_t)\delta^t(\,U_{\de,p}(s',s'/ h_t)-U_{\de,p}(s,s'/ h_t))$ and $U_{\de,p}(s, s)-U_{\de,p}(s',s)= p_{s,s}( h_t)\delta^t(\,U_{\de,p}(s,s'/ h_t)-U_{\de,p}(s',s/ h_t))$, where $p_{s,s}( h_t)$ is the probability that history $h_t$ is reached when both players follow strategy $s$ (note that, given the construction of $s'$, history $h_t$ is reached with the same probability if one or both players follow strategy $s'$). Therefore, the ratio of misscoordination costs is $\frac{U_{\de,p}(s', s'/ h_t)-U_{\de,p}(s, s'/ h_t)}{U_{\de,p}(s, s/ h_t)-U_{\de,p}(s',s/ h_t)}$. Note that, from history $h_t$ onwards, strategy $s$ is identical to Always Defect and strategy $s'$ is identical to Grim. Therefore, the ratio of miscoordination costs is equal to the ratio of miscoordination cost of Always Defect versus Grim: $\frac{U_{\de,p}(s', s'/ h_t)-U_{\de,p}(s, s'/ h_t)}{U_{\de,p}(s, s/ h_t)-U_{\de,p}(s',s/ h_t)} = \frac{U_{\de,p}(g,g)-U_{\de,p}(a,g)}{U_{\de,p}(a,a)-U_{\de,p}(g,a)}$. From the proof of Lemma \ref{lem: AD no se la banca} we know that this ratio can be made arbitrarily large as Always Defect does not have a large basin of attraction against Grim when $\de$ and $p$ are close to one.

We consider now the case in which $h_t$ is not a symmetric history ($h_t\neq \hat{h}_t$). If $s(\hat{h}_th_\tau)=C$ for some $h_\tau$, then the strategy $s$ is not a perfect public equilibrium, as it would cooperate against a strategy that is defecting for ever, and hence it has no basin of attraction against an alternative strategy that defects for ever after $\hat{h}_t$. We focus, then, on the case in which $s(\hat{h}_th_\tau)=D$ for every $h_\tau$. Consider an invader strategy, $s'$ which behaves like $s$ in every history but the histories $h_t$ and $\hat{h}_t$, and those following them. In those histories, $s'$ behaves as a Grim strategy that disregards what happened before $h_t$ or $\hat{h}_t$ (that is, $s(h_t)=s(\hat{h}_t)=C$ and will cooperate in every history $h_th_\tau$ or $\hat{h}_th_\tau$ if, and only if, $h_\tau$ does not include a defection).

Since strategies $s$ and $s'$ only differ in histories $h_t$ and $\hat{h}_t$ and following ones, the ratio of misscoordination costs is:
\begin{eqnarray*}
& &\frac{U_{\de,p}(s', s')-U_{\de,p}(s, s')}{U_{\de,p}(s, s)-U_{\de,p}(s',s)}= \\
& &\frac{p_{s,s}( h_t)\delta^t(\,U_{\de,p}(s',s'/ h_t)-U_{\de,p}(s,s'/ h_t))+p_{s,s}(\hat{h}_t)\delta^t(\,U_{\de,p}(s',s'/ \hat{h}_t)-U_{\de,p}(s,s'/ \hat{h}_t))}{p_{s,s}( h_t)\delta^t(\,U_{\de,p}(s,s/ h_t)-U_{\de,p}(s',s/ h_t))+p_{s,s}(\hat{h}_t)\delta^t(\,U_{\de,p}(s,s/ \hat{h}_t)-U_{\de,p}(s',s/ \hat{h}_t))}.
\end{eqnarray*}

Given that, from histories $h_t$ and $\hat{h}_t$ onwards, strategy $s$ is identical to Always Defect and strategy $s'$ is identical to Grim, $U_{\de,p}(s',s'/ h_t)=U_{\de,p}(s',s'/ \hat{h}_t)=U_{\de,p}(g,g)$, $U_{\de,p}(s,s'/ h_t)=U_{\de,p}(s,s'/ \hat{h}_t)=U_{\de,p}(a,g)$, $U_{\de,p}(s',s/ h_t)=U_{\de,p}(s',s/ \hat{h}_t)=U_{\de,p}(g,a)$, and $U_{\de,p}(s,s/ h_t)=U_{\de,p}(s,s/ \hat{h}_t)=U_{\de,p}(a,a)$. 
Note that $p_{s,s}(h_t)=p_{s,s}(\hat{h}_t)$. Therefore, the ratio of misscoordination costs is equal to the ratio from Always Defect against Grim:
\begin{eqnarray}
\frac{U_{\de,p}(s', s')-U_{\de,p}(s, s')}{U_{\de,p}(s, s)-U_{\de,p}(s',s)} = \frac{U_{\de,p}(g,g)-U_{\de,p}(a,g)}{U_{\de,p}(a,a)-U_{\de,p}(g,a)},
\end{eqnarray}
which, from the proof of Lemma \ref{lem: AD no se la banca}, we know it can be arbitrarily large for $\de$ and $p$ close to one.
\end{proof}

While Grim plays an important role as the alternative strategy in the proof of Theorem \ref{teo: unforgiving}, it is itself an unforgiving strategy. As such, it cannot have a uniformly large basin of attraction. 

Unforgiving strategies are a extreme case of inefficient strategies. In the next section we study the connection between the efficiency of a strategy and the size of its basin of attraction.

\section{Efficiency and size of the basin of attraction}\label{sec: efficiency}

Given a history $h_t$, and a pair of strategies $s, s'$ we define $$U(s,s'/ h_t)=\lim_{\de\to 1} \lim_{p\to 1} U_{\de,p}(s,s'/h_t).$$

\begin{defi} We say that a strategy $s$ is asymptotically efficient if $U(s,s/ h_t)=R$ for any history $h_t.$
\end{defi}

In the next section we prove that symmetric strategies with a uniformly large basin of attraction are asymptotically efficient. In Section \ref{wef-sec} we study non-symmetric strategies and show that for strategies with a uniformly large basin of attraction there is a relation between the ``degree of symmetry'' and payoffs.

\subsection{Symmetric strategies and efficiency}\label{symm sec}

\begin{defi}\label{simetry} We say that a strategy $s$ is symmetric if $s(h_t)=s(\hat h_t)$ for any history $h_t.$
\end{defi}

If the strategy $s$ is symmetric, the pair $(s,s)$ would be a strongly symmetric profile as it is usually defined in the literature (see Fudenberg and Tirole 1991).%
\footnote{We drop the use of the word ``strongly'' for simplicity and given that we use ``symmetry'' to refer to the strategy and not the profile of strategies.} Commonly discussed strategies like Always Defect, Always Cooperate, Grim and Win-Stay-Lose-Shift are symmetric. Tit-for-tat is not symmetric.

\begin{teo}\label{teo: eficiency2} If $s$ has a uniformly large basin of attraction and is symmetric, then it is asymptotically efficient.
\end{teo}

The proof of this theorem uses the large number of alternative strategies present in infinitely repeated games, which was a usual hurdle for the study of evolutionary stability in infinitely repeated games. We construct a sequence of alternative strategies against which an inefficient and symmetric strategy cannot have a uniformly large basin of attraction. For a strategy $s$ to have a uniformly large basin of attraction, it must be that the ratio of cost of miscoordination $\frac{U_{\delta p}(s',s')-U_{\delta p}(s,s') }{U_{\delta p}(s,s)
-U_{\delta p}(s',s)}$ is uniformly bounded for large $\delta$ and $p$ for any alternative strategy $s'$. We construct the alternative strategy $s'$ by making it cooperate forever against itself starting from history $h_{t}$ and its mirror history $\hat h_{t}$ in which $s$ is inefficient, but $s'$ imitates $s$ in other histories (cooperation at $h_t$ by $s'$ works as a ``secrete handshake" that secures future cooperation when matched with itself). Then, the ratio of miscoordination costs that must be bounded is $\frac{U_{\delta p}(s',s'|h_{t}) +U_{\delta p}(s',s'|\hat{h}_{t}) -U_{\delta p}(s,s'|h_{t}) -U_{\delta p}(s,s'|\hat{h}_{t})}{U_{\delta
p}(s,s|h_{t}) +U_{\delta p}(s,s|\hat{h}_{t}) -U_{\delta p}(s',s|h_{t}) -U_{\delta p}(s',s|\hat{h}_{t}) }$. The difference in payoffs must include the history $\hat{h}_{t}$ as $s'$ must differ from $s$ also in $\hat {h}_{t}$. The symmetry of $s$ allows us to prove that $U(s,s^{\prime }|h_{t})+U(s,s'|\hat{h}_{t})=U(s',s|h_{t})+U(s',s|\hat{h}_{t})$. For the ratio of miscoordination costs to be bounded, we must have that the subtracting terms must be even lower than the inefficient payoffs of $s$ at histories $h_{t}
$ and $\hat{h}_{t}$. Since $s'$ imitates $s$ outside of $h_{t}$, this implies that there exists another history $h_{t}^{\prime }$ in which $s$ obtains even lower payoffs than in the original history. Repeating the previous reasoning across a sequence of histories and alternative strategies, we find that $s$ should be increasingly inefficient up to an impossibly low continuation payoff, reaching a contradiction.

{\em Proof of Theorem \ref{teo: eficiency2}:} Assume that there exists a history $h_t$ and a scalar $\la_0<1$ such that $U(s,s/h_t)= \la_0R$. We consider first the case when $h_t$ is not symmetric: $h_t \neq \hat h_t$. Then we show how to deal with the symmetric case using the asymmetric one.

From the fact that $s$ is symmetric, it follows that $U(s,s/h_t)=U(s,s/\hat h_t)$ and, hence, $U(s,s/h_t)+U(s,s/\hat h_t) = 2\la_0R.$
Moreover, since $U(s,s/h_t)< R$ and $s$ is symmetric, we can assume without loss of generality that $s(h_t)=D.$

We chose a strategy  $s'$ such that \emph{i.} $s'(h_t)=s'(\hat h_t)=C$ (the ``secret handshake"); \emph{ii.} $s'$ plays $C$ for ever after $h_t(C,C)$ and $\hat h_t(C,C)$ (responding to the secret handshake); and \emph{iii.} $s'$ imitates $s$ in all other histories.

For the strategy $s$ to have a uniformly large basin of attraction, it must be that the ratio of cost of miscoordination $\frac{U_{\delta p}(s',s')-U_{\delta p}(s,s') }{U_{\delta p}(s,s)-U_{\delta p}(s',s)}$ is uniformly bounded for large $\delta$ and $p$ by, say, a bound equal to $C_0$.
Given the characteristics of $s'$, this ratio of miscoordination is equal to $\frac{U_{\delta p}(s',s'|h_{t}) +U_{\delta p}(s',s'|\hat{h}_{t}) -U_{\delta p}(s,s'|h_{t}) -U_{\delta p}(s,s'|\hat{h}_{t})}{U_{\delta p}(s,s|h_{t}) +U_{\delta p}(s,s|\hat{h}_{t}) -U_{\delta p}(s',s|h_{t}) -U_{\delta p}(s',s|\hat{h}_{t})}.$
The difference in payoffs includes the history $\hat{h}_{t}$ as $s'$ differs from $s$ also in $\hat {h}_{t}$.
Given that both the numerator and denominator are finite, and the denominator must be bounded away from zero (otherwise $s$ would not have a uniformly large basin of attraction and the proof ends), the limit of the ratio must be equal to the ratio of the limits and we focus on the latter in what follows.

Given that $s'$ plays $C$ for ever after $h_t(C,C)$ and $\hat h_t(C,C)$, if follows that $h_{s',s'/h_t}=h_{s',s'/\hat h_t}=(C,C).. (C,C)..,$ and $U(s',s'/h_t)=U(s',s'/\hat h_t)=R.$
From the definition of $s'$ and the symmetry of $s$, it follows that $h_{s,s'/h_t(D,C)}= h_{s',s/\hat h_t(C,D)}$ and $h_{s',s/h_t(C,D)}= h_{s,s'/\hat h_t(D,C)}$. Therefore, $U(s,s'/h_t)=U(s',s/\hat h_t)$ and $U(s',s/h_t)=U(s,s'/\hat h_t)$ (see Figure \ref{teo2 fig}). These equalities imply that $U(s',s/h_t)+U(s',s/\hat h_t)=U(s,s'/h_t)+U(s,s'/\hat h_t).$
Since $(s,s)$ is a strict perfect public equilibrium (otherwise it would not have a uniform large basin of attraction), it follows that $U(s', s/h_t)+U(s', s/\hat h_t)< 2\la_0 R$ and therefore $U(s, s'/h_t)+U(s, s'/\hat h_t)< 2\la_0 R$.

\begin{figure}[!htbpp]
\begin{center}
\includegraphics[width=.7\textwidth]{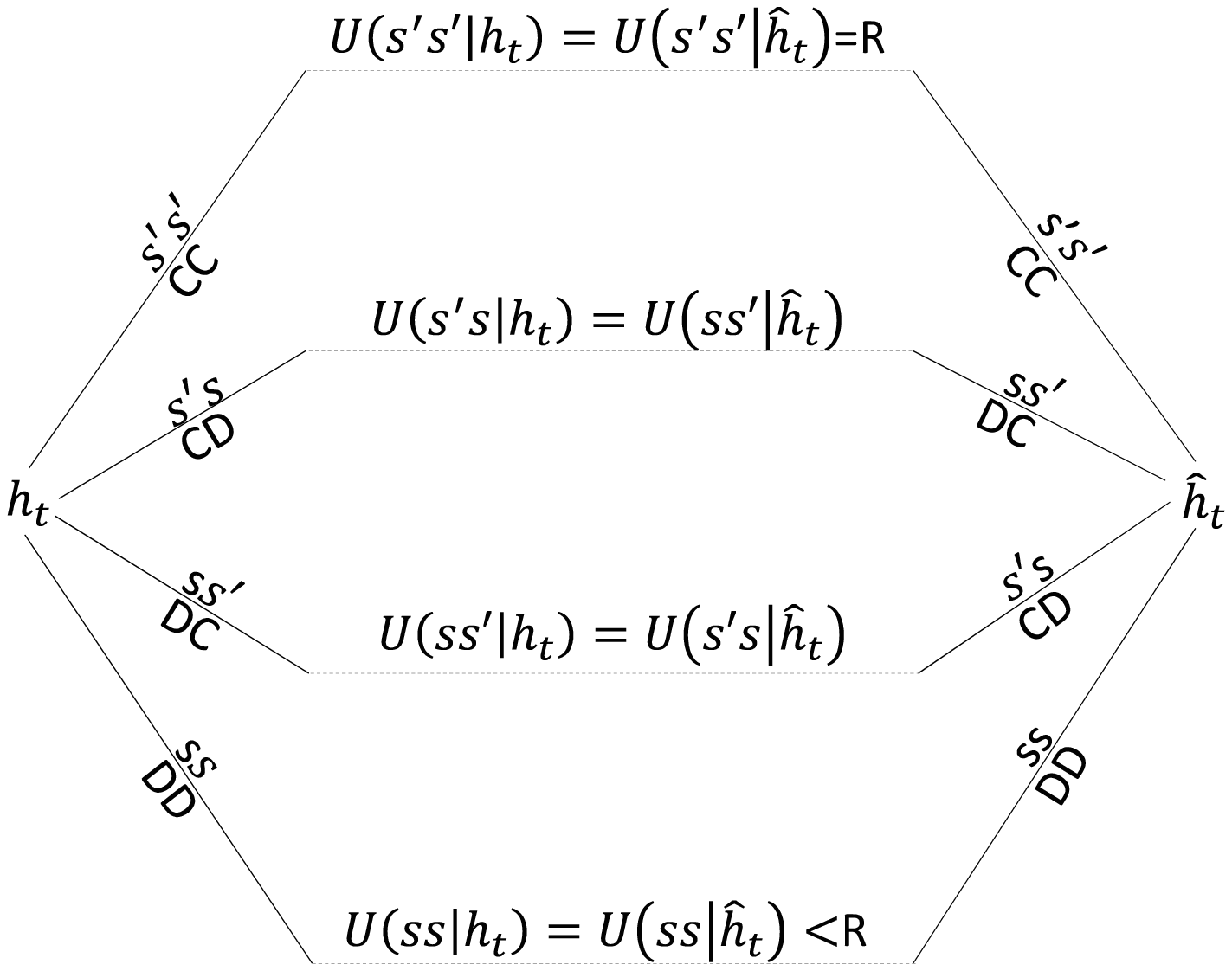}
 \end{center}
\caption{``Secret handshake" and continuation payoffs}
\label{teo2 fig}
\end{figure}

If we denote $U(s',s/h_t)+U(s',s/\hat h_t)=2\la_{1}' R,$ then $\frac{1-\la_{1}'}{\la_0-\la_{1}'}< C_0$ for $s$ to have a uniformly large basin of attraction, and taking a positive constant $C_1< 1-\la_{1}'$ it follows that $\la_{1}'$ satisfies $\frac{C_1}{\la_0-\la_{1}'}< C_0.$ Therefore, there exists $\ga>0$ such that $\la_{1}'< \la_0-\ga.$ 
Now, of the histories $h_t(C,D)$ and $\hat h_t(C,D)$ we take the one with the lowest continuation payoff and denote that history as $h_{t_2}$. We define the number $\la_1$ such that $U(s,s'|h_{t_2})=\la_1 R$. By the choice of $h_{t_2}$, we have that $\la_1 \leq \la_1'$. As before we construct a new strategy $s_2'$ that satisfies the same type of properties as the one satisfied by $s'$ respect to $s$ but on the path $h_{t_2}$ instead of the path $h_{t}.$ Inductively, we construct a sequence of paths $h_{t_i}$, strategies $s_i'$ and constants $\la_i$ and $\la_i'$ such that $U(s,s/h_{t_i})=\la_{i-1}R$, $U(s'_i,s/h_{t_i})+U(s'_i,s/\hat {h}_{t_i})=2\la_i' R$, $\la_i \leq \la_i'$. Arguing as before, it follows that $ \frac{1-\la_{i+1}'}{\la_i-\la_{i+1}'}< C_0,$ and, hence, $\la_{i+1}\leq \la_{i+1}'< \la_i-\ga$. This implies that $\la_{i+1}< \la_1-i\ga$ and $\la_i\to -\infty$. Hence, $ U(s,s/ h_{t_i})\to -\infty$, which is a contradiction because utilities are bounded below by $S.$

To finish, we have to deal with the case that $h_t$ is symmetric and   $U(s,s/ h_t) < R $. Recall that  we can assume that $s(h_t)=s(\hat h_t)=D$. Now, let us consider the history $h_t(C,D)$. We claim that if $U(s,s/ h_t) < R $ then $U(s,s/ h_t(C,D))< R.$
In fact, we can consider the strategy $s'$ such that only differs on $h_t$ and after that plays the same as $s$ plays. Since $(s,s)$ is a strict perfect public equilibrium (otherwise it would not have a uniform large basin of attraction), it follows that $U_{\de, p}(s,s/h_t)>U_{\de, p}(s',s/h_t)$. Therefore, 
$U(s,s/h_t)=\lim_{\de\to 1}\lim_{p\to 1}U_{\de, p}(s,s/h_t)\geq  \lim_{\de\to 1}\lim_{p\to 1}U_{\de, p}(s',s/h_t),$ and, given that 
$\lim_{\de\to 1}\lim_{p\to 1}U_{\de, p}(s',s/h_t)=\lim_{\de\to 1}\lim_{p\to 1}U_{\de, p}(s,s/h_t(C,D))=U(s,s/h_t(C,D))$, the claim follows.  Observe that the new path $h_t(C,D)$ is not symmetric and since the payoff along that path satisfies $U(s,s/ h_t(C,D))< R$, and we argue as above to conclude the proof of Theorem \ref{teo: eficiency2}.
\qed

\subsection{Non-symmetric strategies and efficiency}\label{wef-sec} 
 
The fact that $U_{\delta p}(s,s'|h_{t}) +U_{\delta p}(s,s'|\hat{h}_{t}) =U_{\delta p}(s',s|h_{t}) +U_{\delta p}(s',s|\hat{h}_{t})$ when $s$ is symmetric played a crucial role in the proof of efficiency of symmetric strategies with uniformly large basin of attraction. When $s$ is not symmetric, $U_{\delta p}(s,s'|h_{t}) +U_{\delta p}(s,s'|\hat{h}_{t})$ does not have to equal $U_{\delta p}(s',s|h_{t}) +U_{\delta p}(s',s|\hat{h}_{t})$ and the proof from the previous section cannot be used. Fortunately, we can bound the difference between these two sums of payoffs as a function of the asymmetry of $s.$ This, in turn, allows us to bound from below the payoffs of strategies with a uniformly large basin of attraction as a function of their degree of asymmetry.

Define $u^{\tau}(s,s/h_t)$ as the payoff that a player obtains at time $\tau$ in the 0-tremble history that starts at $h_t$ when both player folow strategy $s$. The next definition is related to the asymmetry of a strategy. In few words, it measures how frequently it happens that $s(h_t)\neq s(\hat h_t)$ along histories.

\begin{defi} A strategy $s$ is $c-$asymmetric if $$c=\lim_{\de\to 1} \sup_{h_{t}} (1-\delta)\left( \sum_{\tau:u^{\tau}(s,s/h_t)=T}\de^{\tau-t}+ \sum_{\tau:u^{\tau}(s,s/h_t)=S}\de^{\tau-t}\right).$$
\end{defi}
Note that a symmetric strategy is 0-asymmetric.

\begin{teo}\label{teowef-1} If $s$ has a uniformly large basin of attraction and is $c-$asymmetric, then 
$U(s,s/h_t)+U(s,s/\hat h_t)\geq 2R-2c(T-S)$ for any $h_{t}.$
\end{teo}

The proof of this theorem is presented in Section \ref{proofs:non-symmetric} in the Appendix.
In this section and the following ones, we relegate proofs to the Appendix as they usually involve several intermediate steps and are more involved than the proofs in the previous sections.

\begin{cor}If $s$ has a uniformly large basin of attraction and is $c-$asymmetric, then 
$U(s,s/h_t)\geq R-c(T-S)$ for any $h_{t}=\hat h_{t}.$
\end{cor}

This corollary is of particular importance as it applies to $h_0,$ bounding the equilibrium payoffs from the beginning of the repeated game for any strategy with a uniformly large basin of attraction.
When strategy $s$ has a uniformly large basin of attraction, there is a minimum bound on the payoffs from $(s,s)$ which is increasing in the degree of symmetry of that strategy.

Note that this result does not imply that it is possible to have inefficient strategies with a uniformly large basin of attraction. It may be the case that our lower bound to payoffs is not tight. Future work should either provide and example of an inefficient and asymmetric strategy with a uniformly large basin of attraction, or show that strategies with a uniformly large basin of attraction must be efficient regardless of symmetry.

\section{Sufficient conditions for a uniformly large basin}
\label{large basin}

In this section we provide general sufficient conditions for a strategy to have a uniformly large basin of attraction.

First, based on properties of replicator dynamics, we show that if a strategy satisfies a condition involving all possible \emph{pairs} of invading strategies, then it has a uniformly large basin of attraction. As this condition involves working with three strategies at a time, we call it the \emph{trifecta} condition. Second, we provide a even simpler condition for a strategy to have a uniformly large basin of attraction which requires only working with two strategies at a time. We call this condition the \emph{cross ratio} condition.

\subsection{The trifecta condition}

Let $s$ be a strict perfect public equilibrium strategy for $\de$ and $p$ large. Given $s'$ and $s^*$ with $U_{\de,p}(s,s)-U_{\de,p}(s^*,s)\geq U_{\de,p}(s,s)-U_{\de,p}(s',s)$, we define the following number
\begin{eqnarray}\label{Ms}
 M_{\de,p}(s,s^*,s'):=\frac{U_{\de,p}(s',s^*)-U_{\de,p}(s, s^*)+U_{\de,p}(s^*,s')-U_{\de,p}(s, s'),}{U_{\de,p}(s,s)-U_{\de,p}(s^*,s)}.
\end{eqnarray}
We consider the supremum of $M_{\de,p}(s,s^*,s')$ for all $s', s^*$:
\begin{eqnarray}\label{M0}
M_{\de,p}(s):=\sup_{ U_{\de,p}(s,s)-U_{\de,p}(s^*,s)\geq U_{\de,p}(s,s)-U_{\de,p}(s',s)} \{M_{\de,p}(s,s^*,s'), \,\,\, 0 \}.
\end{eqnarray}
 Taking the limit on $\de$ and $p$, we  also define
\begin{eqnarray}\label{M1}
M(s):=\limsup_{{\de}\to 1}\limsup_{{p}\to 1}M_{\de,p}(s).
 \end{eqnarray}
Note that if $M(s)$ is finite, then there exist $M_0, \de_0>0$ and an increasing function function $p:[\de, 1]\to [0,1]$ such that $M_{\de, p}(s)< M_0$ for $\de>\de_0$ and $p> p(\de).$

\begin{defi}\label{ULBS cond}
We say that a strategy $s$ satisfies the ``trifecta condition'' if $M(s)<\infty.$
\end{defi}

\begin{teo}\label{ULBA for games} If $s$ satisfies the trifecta condition and $(s,s)$ is a strict perfect public equilibrium strategy, then $s$ has a uniformly large basin of attraction.
\end{teo}

The proof of this theorem is presented in Section \ref{proofs:sufficientconditions} in the Appendix.

The reason that we only need to study all possible combinations of two invading strategies is related to the fact that the replicator dynamic is monotone with respect to the difference between the average payoff of a fixed strategy $s$ against any strategy in a finite population and the average payoff in that population. Both quantities are given by the average payoff of the fixed strategy against any other one and the average payoff between any other pair of strategies. Using the linearity of the average, everything is reduced to comparing payoffs involving $s$ and any other pair in that population. In particular, if it is possible to uniformly bound the relation between a fixed strategy and any other pair, it is possible to bound the averages. 

\subsection{The cross ratio condition}

In the previous section we showed that it is enough to check a condition involving invasion by all possible combinations of two strategies to verify that a strategy has a uniformly large basin of attraction. In this section we provide a further simplification by showing that it is enough to check a condition involving invasions by single strategies under certain conditions. This greatly reduces the complexity of verifying that a strategy has a uniformly large basin of attraction.

\begin{defi} \label{sgpstrict}  We say that $(s,s)$ is a uniformly strict perfect public equilibrium if there exist $0< \de_0<1$, an increasing continuous function $p:[\de_0, 1]\to [0,1]$ (with $p(1)=1$) and a positive constant $C_0$  that only depends on $T, R, P, S$  such that $U_{\de, p}(s,s/h_t)-U_{\de, p}(s',s/h_t)> C_0(1-p^2\de),$ for any $h_t$ such that $s(h_t)\neq s'(h_t)$, any strategy $s'$, any $\de>\de_0$ and $p>p(\de).$ 
\end{defi}
In short, we will not only request that $(s,s)$ is a strict perfect public equilibrium, but that the difference $U_{\de, p}(s,s/h_t)-U_{\de, p}(s',s/h_t)$ can not be smaller than a decreasing function of the discount factor. 

\begin{defi} \label{effstrict}  We say that $(s,s)$ is uniformly efficient if there exist $0< \de_0<1$, an increasing continuous function $p:[\de_0, 1]\to [0,1]$ (with $p(1)=1$) and a positive constant $C_1$  that only depends on $T, R, P, S$ such that $R- U_{\de, p}(s,s/h_t)< C_1(1-p^2\de),$ for any $h_t$, $\de>\de_0$ and $p>p(\de).$
\end{defi}

In short, we say that $(s,s)$ is an {\it uniformly efficient strict perfect public equilibrium} if it satisfies the previous two definitions.

We introduce next a technical condition that only involves two strategies, and therefore is easier to check than the trifecta condition. As we will see, this technical condition, which we call the ``cross ratio condition,'' implies the trifecta condition under some additional assumptions.

\begin{defi} \label{***} We say that $s$ satisfies the cross ratio condition if there exist $0< \de_0<1$, an increasing continuous function $p:[\de_0, 1]\to [0,1]$ (with $p(1)=1$) and a positive constant  $C_2$ that only depends on $T, R, P, S,$ such that   
\begin{eqnarray}\label{totalpath}
 \frac{U_{\de,p}(s,s)-U_{\de,p}(s,s')}{U_{\de,p}(s,s)-U_{\de,p}(s',s)}< C_2,
\end{eqnarray}
for any $\de>\de_0$ and $p>p(\de)$ and any  strategy $s'$.
\end{defi}
The cross ratio condition relates how $s'$ performs against $s$ with how $s$ performs against $s'.$
Put simply, this condition says that $U_{\de,p}(s,s')$ cannot be much below $U_{\de,p}(s,s)$ if $U_{\de,p}(s',s)$ is close  to $U_{\de,p}(s,s)$.
In other words, this conditions puts a lower bound to how well stsrategy $s$ does against any other strategy $s'$ as a function of how well $s'$ does against $s$.%
\footnote{Observe that the cross ratio condition resembles the condition required for $s$ to have a uniformly large basin of attraction agains an invasion by a single strategy $s'$ (see condition (\ref{eq 2 a 2})). Note, however that the numerators are different: it is $U_{\de,p}(s',s')-U_{\de,p}(s,s')$ in the latter, and $U_{\de,p}(s,s)-U_{\de,p}(s,s')$ in the former.}

We show next that, if a strategy is a uniformly efficient strict perfect public equilibrium and satisfies the cross ratio condition, then it has a uniformly large basin of attraction.

\begin{teo}\label{2 implica ULBC}If $s$ satisfies the cross ratio condition and $(s,s)$ is a uniformly efficient strict perfect public equilibrium, then $s$ has a uniformly large basin.
\end{teo}

While it is intuitive that an efficient strategy that earns high payoffs when playing against any other strategy relative to the payoff that any other strategy earns against it, the proof requires several intermediate results and it is  provided in Section \ref{two to ulbc} in the Appendix.
The proof consists of showing that such a strategy satisfies the trifecta condition and, hence, it has a uniformly large basin of attraction by Theorem \ref{ULBA for games}.

\section{Strategies with uniformly large basins of attraction}
\label{W banca tr subs}

In this section we study the existence of strategies with a uniformly large basins of attraction. We present a large class of strategies that satisfy the conditions introduced in Section \ref{large basin} and, hence, have a uniformly large basin of attraction. We then show that some commonly described strategies belong to this class and have uniformly large basins of attractions.

\subsection{Star-type strategies}
\label{type w}

The main property that star-type strategies satisfy is based on the frequency that someone following the strategy spends in each of the outcomes of the game when playing against another strategy. Remember that we denote by $h^{\tau}_{s,s'/h_t}$ for $\tau>t$ the 0-tremble actions at time $\tau$ if players follow strategies $s$ and $s'$ and history $h_t$ has been reached.

\begin{defi}\label{frequencies}
Given two strategies $s,s'$ and a finite history $h_t$ we define
 $$b_R=\frac{1-p^2\de}{p^2}\sum_{\tau: u(h{^\tau}_{s',s/h_{t}})=R}\, p^{2(\tau-t)+2}\de^{\tau-t},\,\,b_S=\frac{1-p^2\de}{p^2}\sum_{\tau: u(h{^\tau}_{s',s/h_{t}})=S}\, p^{2(\tau-t)+2}\de^{\tau-t},$$
$$b_T=\frac{1-p^2\de}{p^2}\sum_{\tau: u(h{^\tau}_{s',s/h_{t}})=T}\, p^{2(\tau-t)+2}\de^{\tau-t},\,\, b_P=\frac{1-p^2\de}{p^2}\sum_{\tau: u(h{^\tau}_{s',s/h_{t}})=P}\,p^{2(\tau-t)+2} \de^{\tau-t}.$$

Observe that the quantities $b_R, b_S, b_T,$ and $b_P$ depend on $s, s', \de, p$ and $h_t.$ 
\end{defi}

In short, the quantities defined above are the discounted frequencies of earning R, S, T, and P along the $0-$tremble history of $(s',s)$ starting at $h_t.$
\begin{defi}\label{type w defi} We say that  $s$ is a star-type strategy if
\begin{enumerate}
\item $(s,s)$ is a uniformly efficient strict perfect public equilibrium;
\item there exist $\de_0, p_0, C_5$ such that for any strategy $s'$, any $\de> \de_0, p>p_0$ and any finite history  $h_t$  such that $s(h_t)\neq s(h_t)$ it follows that if $b_T>0$ then
\begin{eqnarray}\label{eq on b}
 b_T< \ga_2 b_S +\ga_4 b_P+C_5(1-p^2\de),
\end{eqnarray}
\end{enumerate}
 where $\ga_2=\frac{R-S}{T-R}$, $\ga_4=\frac{R-P}{T-R}$ and $b_S, b_T, b_P$ are the quantities given by Definition \ref{frequencies}. We call inequality (\ref{eq on b}), the sufficiently responsive condition.
\end{defi}

The sufficiently responsive condition is related to  the frequency of playing cooperation or defection  towards $s'$ along the $0-$tremble history starting at any finite history. In fact, if at some time  $s'$ defects while $s$ cooperates (in  other words $s'$ scores $T$), then $b_T$ is positive and so the sufficient responsive condition implies that either $b_S$ or $b_P$ are also positive, that is, $s$ has to defect. 
In short, we can say that $s$ retaliates against defections by $s'$.
  
The next theorem shows that star-type strategies have a large basin of attraction.
\begin{teo} \label{* type ULBA} Any star-type strategy has a uniformly large basin of attraction.
\end{teo}

The proof is provided in Section \ref{ap for sect6} of the Appendix. The proof consists, mainly, on showing that a star-type strategy satisfies the cross ratio condition. The intuition is that an efficient strategy that satisfies the sufficiently responsive condition (see condition \ref{eq on b}) does not let other strategies take advantage of it, and as such, it will do well against them relative to how any other strategy performs against it (satisfying the cross ratio condition).

\subsection{Existence}
\label{w subsection}

In this section we provide examples of strategies with uniformly large basin of attraction.

We start by showing that, under certain condition on the payoff matrix, the strategy win-stay-lose-shift (WSLS or $w$) has a uniformly large basin of attraction.\footnote{The strategy win-stay-lose-shift (WSLS or $w$) is also known as Perfect TFT and was introduced by Fudenberg and Tirole (1991).} We define this strategy next.

\begin{defi}{\bf win-stay-lose-shift (w)}
Cooperates in $t=0$, and in $t>0$ cooperates if it earned either $R$ or $P$ in $t-1$, and defects otherwise.
\end{defi}

This strategy is described as a two-state machine in Figure \ref{strategies fig}. The strategy starts in the cooperation state, and stays in that state if both players cooperate or defect: $CC$ and $DD$, where the first letter denotes the action of the player and the second letter the action of the other player. The strategy moves to the defection state if one of the two players defects: $DC$ or $CD$. If the strategy is in the defection state, it returns to the cooperation state if both players cooperate or defect, $CC$ or $DD$, and remains in that state if only one of them defects: $CD$ and $DC$. In particular, observe that $w$ punishes any defection (which we will show that it implies that $w$ satisfies the sufficiently responsive condition) but it does not necessary return to cooperation, in fact it keeps defecting if it reached the defecting state and the other strategy cooperates (it could be say that $w$ may take advantage of the other player). Note also that $w$ punishes deviations with only one period of defection. Hence, $(w,w)$ can only be a subgame perfect equilibrium for sufficiently large $\delta$ if $2R>T+P$.

\begin{figure}[!htbpp]
\begin{center}
\includegraphics[width=.7\textwidth]{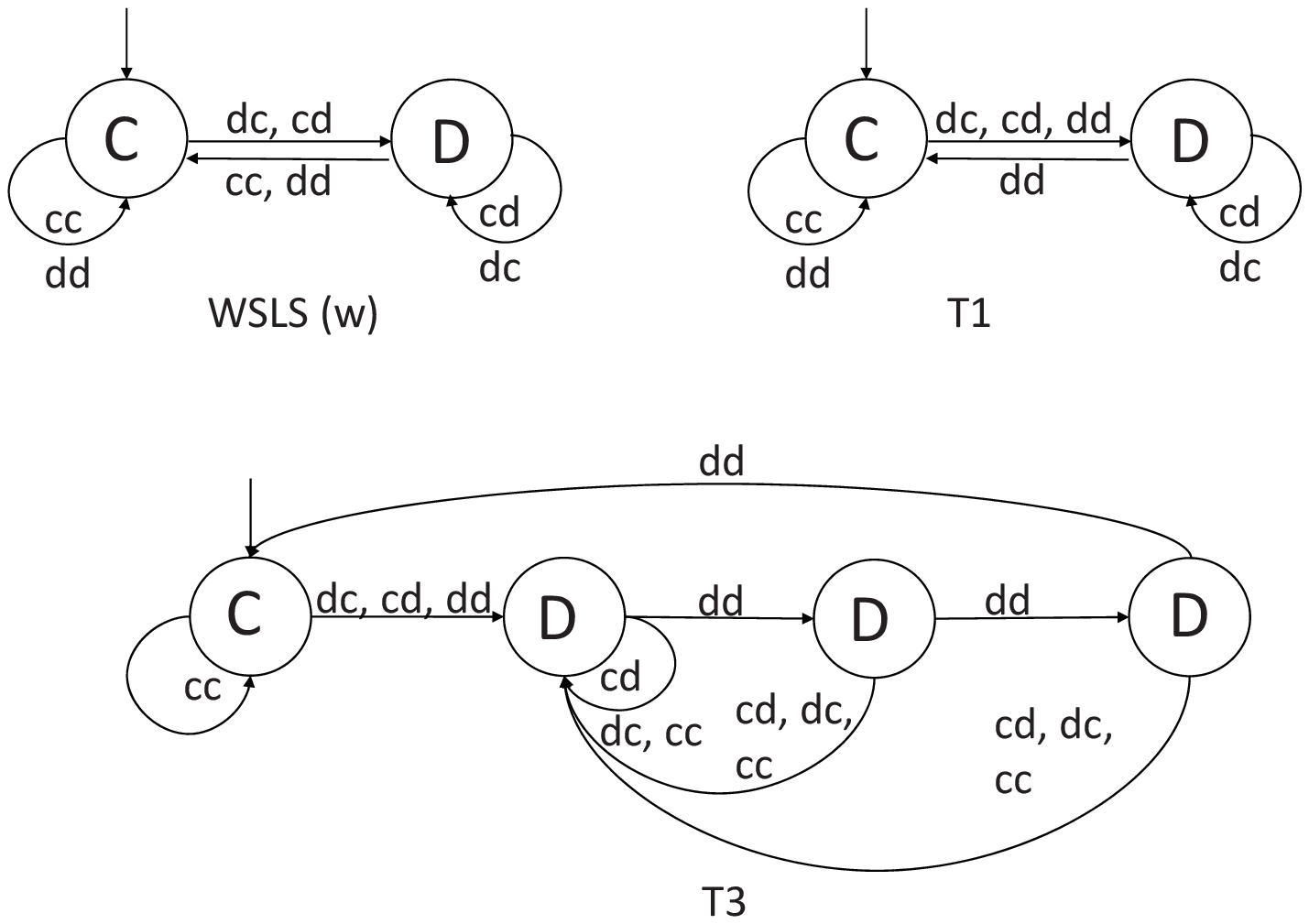}
 \end{center}
\caption{Examples of star-type strategies.}
\label{strategies fig}
\end{figure}

\begin{teo}\label{w-teo} Win-stay-lose-shift ($w$) has a uniformly large basin of attraction if $2R>T+P$.
\end{teo}

The proof of this theorem consists of showing that $w$ is a star-type strategy under that payoff restriction. The proof is provided in Section \ref{w teo proof} of the Appendix.

We show next, that even if this condition on payoffs does not hold, there are other strategies with a uniformly large basin of attraction. For that, we define next the family of $n$ trigger strategies ($Tn$), which react to a deviation from the prescribed behavior of the strategy by having $n$ periods of defection.

\begin{defi}{\bf \emph{n} Trigger ($Tn$)}
At $t=0$, it cooperates. At $t>0$, it cooperates if there was no deviation from the action of the strategy by any of the two players in the last \emph{n} periods, otherwise it defects. 
\end{defi}

Figure \ref{strategies fig} shows $T1$ and $T3$ as 2-state and 4-state machines. $T1$ starts in the cooperation state and stays there if both players cooperate $(CC)$, it moves to defection if there is a defection ($DC$, $CD$ or $DD$). Once in the defection state, the strategy moves to the cooperation state if both players defect ($DD$) and stays in the defection state otherwise ($CD$, $DC$, $CC$). $T3$ is similar except that there are more periods of punishment after a deviation. 

Note that there is a very small difference between $w$ and $T1$: the latter stays in the punishment stage if both players cooperate when they should be defecting. However, for $n>1$, $Tn$ allows for stronger punishments than $w$ (it punishes a deviation with more periods of defection). This stronger punishment allows $(Tn,Tn)$ to be a sub-game perfect equilibrium for a greater set of payoff paremeters than $(w,w)$. This also leads to the existence of strategies with uniformly large basin of attraction without constraints in the payoff parameters.

\begin{teo}\label{Tn-teo} For any Prisoner's Dilemma game, there exists a large enough $n$ such that $Tn$ has a uniformly large basin of attraction.
\end{teo}

The proof of this theorem is provided in Section \ref{Tn is sgp} in the Appendix and consists of showing that $Tn$ is a star-type strategy under that payoff restriction.

\section{Conclusions}

There is an extensive literature on the evolutionary determinants of cooperation in repeated games starting with the work by Axelrod and Hamilton (1981) and Axelrod (1984). We contribute to this literature by studying the evolutionary robustness of strategies in infinitely repeated prisoners' dilemma games with arbitrarily patient players and small probability of mistakes. We show that there are strategies which can repel invasion of up to a uniform size by any finite combination of invading strategies. These strategies cannot be unforgiving and they must cooperate if they are symmetric. We show that there is a large class of strategies that have a uniformly large basing of attraction. Examples of such a strategies are win-stay-lose-shift and trigger strategies.

A previous theoretical literature provides evolutionary support for the strategy win-stay-lose-shift (see Nowak and Sigmund 1992, and Imhof, Fudenberg and Nowak 2007). This strategy has received little support from experiments on infinitely repeated games (see Dal B\'o and Fr\'echette 2011, Fudenberg, Rand and Dreber 2012, and Dal B\'o and Fr\'echette 2019). We hope that new experiments can be designed to test whether this strategy, and other strategies with a uniformly large basin of attraction, are robust to invasions when they are already highly prevalent.

\section{Appendix A. Proofs}\label{proofs}

\subsection{Proofs for Section \ref{wef-sec}: Non-symmetric strategies and efficiency}\label{proofs:non-symmetric}
Before proving Theorem \ref{teowef-1} we provide two lemmas relating payoffs at $h_t$ and $\hat h_t.$

\begin{lem}\label{wef-3} For any profile of strategies $(s,s')$ and a history $h_t,$ $U_{\de,p=1}(s',s/\hat h_t)=U_{\de,p=1}(s,s'/ h_t)+(1-\delta)(a_S-a_T)(T-S),$ where
 $a_S=\sum_{\tau:u^{\tau}(s,s'/h_t)=S}\de^{\tau-t}$ and $a_T= \sum_{\tau:u^{\tau}(s,s'/h_t)=T}\de^{\tau-t}.$
\end{lem}

\begin{proof}
Given that $U_{\de,p=1}(s,s'/ h_t)=(1-\delta)(a_R R+a_S S+a_T T+a_P P)$ where the constants $a_R$ and $a_P$ are defined as $a_R=\sum_{\tau:u^\tau(s,s'/h_t)=R}\de^{\tau-t}$ and $a_P= \sum_{\tau:u^{\tau}(s,s'/h_t)=P}\de^{\tau-t}$, then $U_{\de,p=1}(s',s/\hat h_t)=(1-\delta)(a_R R+a_S T+a_T S+a_P P)=U_{\de,p=1}(s,s'/h_t)+(1-\delta)(a_S T+a_T S-a_S S-a_T T)=U_{\de,p=1}(s,s'/h_t)+(1-\delta)(a_S-a_T)(T-S).$
\end{proof}

\begin{lem}\label{wef-4} Given a strategy profile $(s,s')$ and a path $h_t$ it follows that $U_{\de,p=1}(s,s'/\hat h_t)\leq U_{\de,p=1}(s',s/h_t)+(1-\delta)(a_S+a_T)(T-S)$.
\end{lem}

\begin{proof}
From Lemma \ref{wef-3} and the fact that $a_S,a_T\geq0$ and $T-S>0$, $U_{\de,p=1}(s,s'/\hat h_t)=U_{\de,p=1}(s',s/h_t)+(1-\delta)(a_S-a_T)(T-S)\leq U_{\de,p=1}(s',s/h_t)+(1-\delta)(a_S+a_T)(T-S)$.
\end{proof}

\noi{\it Proof of Theorem \ref{teowef-1}.} Assume, by contradiction, that $U(s,s/h_t)+U(s,s/\hat h_t)< 2R-{2}c(T-S)$ for some $h_t$. Consider a particular history $h_t$ such that for some $R_1<R$ it holds  $U(s,s/h_t)+U(s,s/\hat h_t)= 2R_1-{2}c(T-S)$ and this value of payoffs is arbitrarily close to the infimum of all possible values. Define an alternative strategy $s'$ such that: first, $s'$ differs from $s$ in histories $h_t$ and $\hat h_t$: $s'(h_t)\neq s(h_t)$ and $s'(\hat h_t)\neq s(\hat h_t)$; second, $s'$ cooperates with itself after $h_t$ and $\hat h_t$: $s'(h_ts'(h_t)s'(\hat h_t)h_{\tau})=C$ and $s'(\hat h_ts'(\hat h_t)s'(h_t)\hat h_{\tau})=C$ for all $h_{\tau}$; and, third, $s'$ imitates $s$ in all other histories.

We will show that the ratio of miscoordination costs $$\frac{U_{\delta p}(s',s'|h_{t}) +U_{\delta p}(s',s'|\hat{h}_{t}) -U_{\delta p}(s,s'|h_{t}) -U_{\delta p}(s,s'|\hat{h}_{t})}{U_{\delta p}(s,s|h_{t}) +U_{\delta p}(s,s|\hat{h}_{t}) -U_{\delta p}(s',s|h_{t}) -U_{\delta p}(s',s|\hat{h}_{t})}$$ can be made arbitrarily large and, hence, $s$ does not not have a uniformly large basin of attraction.

We focus first on the numerator of the ratio of miscoordination costs. Note that by Lemma \ref{wef-4} we have that 
\begin{eqnarray*}
U_{\de,p=1}(s,s'/h_t)+U_{\de,p=1}(s,s'/\hat h_t) &\leq & U_{\de,p=1}(s',s/h_t)+U_{\de,p=1}(s',s/\hat h_t)\\
&+&(1-\delta)(a_S+a_T+\hat a_S+\hat a_T)(T-S), 
\end{eqnarray*}
 where the $a_S$ and $a_T$ are defined as in Lemma \ref{wef-3} and $\hat a_S$ and $\hat a_T$ are similarly defined but starting from history $\hat h_t$. Given the construction of $s'$ and the definition of $c$ we have that $(1-\delta)(a_S+a_T+\hat a_S + \hat a_T) \leq 2c+2(1-\delta).$ Then $U_{\de,p=1}(s,s'/h_t)+U_{\de,p=1}(s,s'/\hat h_t)\leq U_{\de,p=1}(s',s/h_t)+U_{\de,p=1}(s',s/\hat h_t))+2(c+1-\delta)(T-S)$ which in turn is smaller than $U_{\de,p=1}(s,s/h_t)+U_{\de,p=1}(s,s/\hat h_t)+2(c+1-\delta)(T-S)$ as $(s,s)$ is a strict perfect public equilibrium. Given that $\lim_{\de\to 1}U_{\de,p=1}(s,s/h_t)+U_{\de,p=1}(s,s/\hat h_t)=2R_1-2c(T-S),$ we have that $\lim_{\de\to 1}\left(U_{\de,p=1}(s,s'/h_t)+U_{\de,p=1}(s,s'/\hat h_t)\right)\leq 2R_1.$

Therefore, the limit of the numerator of the ratio of miscoordination costs is bounded away from zero:
$$\lim_{\de\to 1}[U_{\de,p=1}(s',s'/h_t)+U_{\de,p=1}(s',s'/\hat h_t)-U_{\de,p=1}(s,s'/h_t)-U_{\de,p=1}(s,s'/\hat h_t)]\geq  2(R-R_1).$$

We focus now on the denominator of the ratio of miscoordination costs. It holds that $U_{\de, p}(s,s/ h_t)+U_{\de, p}(s,s/\hat h_t)>U_{\de, p}( s',s/ h_t)+U_{\de, p}(s',s/\hat h_t)$ due to the fact that $s$ is a strict perfect public equilibrium (otherwise it would not have a uniformly large basin of attraction). Moreover, note that the difference is arbitrarily small for large $\delta$ as $U_{\de, p}(s,s/ h_t)+U_{\de, p}(s,s/\hat h_t)$ is, by choice of $h_t$, arbitrarily close to the infimum of the possible equilibrium values of this sum.

Since the denominator of the ratio of miscoordination costs is arbitrarily small but the numerator is bounded away from zero, the ratio of miscoordination costs can be made arbitrarily large and $s$ does not have a uniformly large basin of attraction.\qed

\subsection{Proofs for Section \ref{large basin}: Sufficient conditions to have a uniformly large basin}\label{proofs:sufficientconditions}

In the next subsection we provide a result on replicator dynamics which is the base of the proof of Theorem \ref{ULBA for games}, which we present in Section \ref{newconditions}. Subsection \ref{alcanza de a 2} provides some preliminary results regarding the cross ratio condition, and Section \ref{two to ulbc} provides the proof of Theorem \ref{2 implica ULBC}.

\subsubsection{Main theorem for replicator dynamics}
\label{ULBA for RD}

Consider the replicator dynamics as defined in Section \ref{sec: replicator}. Recall that we consider the replicator dynamics associated to the matrix $U$ on the $n$ dimensional simplex $\Delta=\{x=(x_1\dots x_n)\in \RR^n: x_1+\dots+x_n=1,\,\, x_i\geq 0,\forall i\}$ denoting all the possible distributions of the strategies in the population of agents. The replicator dynamics are given by $ \dot x_i=x_i[u_i(x)-\bar u(x)]$ where $u_i(x) = \sum_{j=1}^n x_ju_{ij}$ is the expected payoff of strategy $i$, and $\,\,\,\bar u(x)=\sum_{j=1}^n x_ju_j(x)$ is the average expected payoff.

We start by simplifying the replicator equations by writing them as a function of the vector $\bar x = (x_2,..,x_n)$. The reason we can do this is that, by definition, $x_1=1-\sum_{i\geq 2}x_i.$ That is, we consider an affine change of coordinates to define the dynamics in the set  $[0,1]^{n-1}$  instead of the simplex $\De$. Observe that in these coordinates the point $(0,.., 0)\in[0,1]^{n-1}$ corresponds to $e_1=(1,0,.., 0)\in[0,1]^{n}$, and $\{(x_2,.., x_n): x_i\geq 0, \sum_{i=2}^n x_i \leq 1\}$ corresponds to the simplex $\De$.

Given the payoff matrix $U$, we define a  matrix $M$ and the vector $N$ given by 
$$N_{i}=u_{11}-u_{i1},\,\,\, M_{ij}= u_{ij}-u_{1j}+u_{11}-u_{i1}.$$
Moreover, we assume that 
the vertex $\{e_2\dots e_n\}$ are ordered in such a way that 
$u_{11}-u_{i1}\geq u_{11}-u_{j1}$ for any $ 2\leq i<j.$

\begin{teo}\label{a1} Let $U\in \RR^{n\times n}$ such that $u_{j1}< u_{11}$ for any $2\leq j\leq n.$ 
\begin{eqnarray}\label{NM2}
  M_0&=&\max_{i,j\geq i}\{\frac{M_{ij}+M_{ji}}{N_i},\,\,\, 0 \}.
\end{eqnarray}
  Then,   $\De_{\frac{1}{M_0}}=\{x: \sum_{i\geq 2}\,  x_i \leq \frac{1}{M_0}\} \subset B_{loc}(e_1)$ (the local basin of attraction of $e_1$). 
\end{teo}

 The intuition behind this result is that the replicator dynamic is a quadratic equation and therefore only pairs of alternative strategies matter in calculating the differences in payoffs. In the appendix we show that this result also holds for more general evolutionary dynamics. 

The proof of Theorem \ref{a1} is based on the next lemma about quadratic polynomials.

\begin{lem}\label{gralA}
 Let $Q:\RR^n\to \RR$ given by 
$Q(x)=-Nx+x^tMx$ with $x\in \RR^n$, $N\in \RR^n$ and $ M\in \RR^{n\times n}$. Let us assume that $N_i>0$, and $N_i\geq N_j$ for any $j>i$. 
Let $M_0=\max_{i,\,j> i}\{\frac{M_{ij}+M_{ji}}{N_i},\,\,\,0\}.$
Then,  the set $\De_{\frac{1}{M_0}}=\{x\in \RR^n: x_i\geq 0, \sum_{i=1}^n x_i <\frac{1}{M_0}\} $ is contained in $ \{x: Q(x)< 0\}.$ In particular, if $M_0=0$ then  $\frac{1}{M_0}$ is treated as $\infty$ and this means that  $\{x\in \RR^n: x_i\geq 0\}\subset \{x: Q(x)\leq 0\}.$
\end{lem}

\begin{proof} For any $v\in \RR^n$ such that $v_i\geq 0$ and $\sum_i\, v_i=1$, we consider the following one dimensional quadratic polynomial, $Q^v:\RR\to \RR$ given by 
$ Q^v(s):= Q(sv)= -s Nv+ s^2 v^tMv.$
To prove the thesis of the lemma it is enough to show the following claim:  for any positive vector $v$  with norm  equal to $ 1$, if $ 0<s<\frac{1}{M_0}$ then $Q^v(s)<0$; in fact, the claim implies the lemma, otherwise, arguing by contradiction,  if there is a point $x_0\in \De_\frac{1}{M_0}$ different than zero (i.e.: $0<|x_0|<\frac{1}{M_0}$) such that $Q(x_0)=0$, then taking $v=\frac{x_0}{|x_0|}$ and $s=|x_0|$ it follows that $Q^v(s)=-Nx_0+x_0^tMx_0=0$, but $|v|=1, s< \frac{1}{M_0}$, a contradiction.

Now we proceed to show the above claim. Given that $Nv= \sum N_iv_i >0,$ if $v^tMv=0,$ then $Q^v(s)<0$. If if $v^tMv\neq0,$ the roots of $Q^v(s)$ are given by $s=0$ and $s=\frac{Nv}{v^tMv}.$
 If $v^tMv<0$ then it follows that $Q^v$ is a one dimensional quadratic polynomial with negative quadratic term and two non-positive roots, so  for any $s>0$ it holds that $Q^v(s)<0$ and therefore proving the claim in this case. So, it remains to consider the case in which $v^tMv>0$. In this case, $Q^v$ is a one dimensional quadratic polynomial with positive quadratic term ($v^tMv$), therefore for any $s$ between both roots ($0, \frac{Nv}{v^tMv}$) it follows that $Q<0$. To finish we have to prove that $\frac{Nv}{v^tMv}\geq\frac{1}{M_0}$ which follows from the next inequalities:
\begin{eqnarray*}
v^tMv&=& \sum_{ij} v_iv_jM_{ij}= \sum_i [v_i^2 M_{ii}+ \sum_{j>i} v_iv_j (M_{ij}+M_{ji})]\\
&\leq& \sum_i [v_i^2 N_iM_0 + \sum_{j>i} v_iv_j N_iM_0]= M_0 \sum_i N_i v_i [\sum_{j\geq i} v_j]\\
&\leq&  M_0 \sum_i N_iv_i= M_0 Nv.
 \end{eqnarray*}
 \end{proof}

\noindent{\em Proof of Theorem \ref{a1}:} Remember that $\dot x_i=x_i[u_i(x)-\bar u(x)]$ by the definition of replicator dynamics. Note that $u_i(x)-\bar u(x)=[u_i(x)-u_1(x)]+[u_1(x)-\bar u(x)]$. The second term in brackets is $u_1(x)-\bar u(x)=\sum_{j\geq 1}\,  x_j[ u_1- u_j]=\sum_{j\geq 2}\,  x_j[ u_1- u_j].$ Therefore, the difference in payoffs determining the growth of strategy $i$ is: $u_i(x)-\bar u(x)=(u_i-u_1)(x)+R(x)$ where $R(x)= \sum_{j\geq 2} \,(u_1-u_j)(x) x_j$.

We consider the change of coordinates $\bar x = (x_2,..,x_n)$ as introduced before.  For any $k< 1$ we denote  $\De_k:= \{\bar x: \sum_{i\geq 2}\,  x_i \leq k\}$ and $\partial \De_k= \{\bar x: \sum_{i\geq 2}\,  x_i = k\}.$ To conclude that for any initial condition in $\De_{\frac{1}{M_0}}$ follows that its forward trajectory converges to zero, it is enough to show that 
the function $\bar X(t):=\sum_{i\geq 2} \, x_i(t)$ is a strictly decreasing function of $t$. In fact, if this holds, given that the trajectory is always in $\De$, then $\bar X(t)\to 0$ and therefore $x_i(t)\to 0$ for any $i\geq 2$. To do that, we prove $\dot{\bar X}=\frac{\partial\bar X}{\partial t}<0$ if $\bar x(t)=(x_2(t),\dots, x_n(t))\in \De_{\frac{1}{M_0}}.$ Therefore, we have to show that
\begin{equation}\label{cond-dec}
Q(\bar x):= \sum_{i\geq 2}\,  \dot x_i < 0.
\end{equation}
Remember that $\dot x_i=x_i[u_i(\bar x)-\bar u(\bar x)]$ by the definition of replicator dynamics, and that $u_i(\bar x)-\bar u(\bar x)=(u_i-u_1)(\bar x)+R(\bar x)$ where $R(\bar x)= \sum_{j\geq 2} \,(u_1-u_j)(\bar x) x_j$.
Therefore, 
\begin{eqnarray*}
 Q(\bar x)&=& \sum_{i\geq 2}\, (u_i-u_1)(\bar x) x_i + \sum_{i\geq 2}\, R(\bar x) x_i=\sum_{i\geq 2}\, (u_i-u_1)(\bar x) x_i +  R(\bar x)\sum_{i\geq 2}\, x_i.
\end{eqnarray*}
 Since $\sum_{i\geq 2}\, x_i=k$ (with $k<1$), it follows that 
$Q(\bar x)= \sum_{i\geq 2}\, (u_i-u_1)(\bar x) x_i +  R(\bar x) k.$ Recalling the definition of $R$, we get that 
$ Q(\bar x)= (1-k)\sum_{i\geq 2}\, (u_i-u_1)(\bar x) x_i.$
So, to prove inequality (\ref{cond-dec}) is enough to show that  
\begin{equation}\label{cond-dec-2}
 Q(\bar x)=  (1-k)\sum_{i\geq 2} \,x_i (u_i-u_1)(\bar x)< 0 
\end{equation}
for any $\bar x\in \Delta_k$ and $ k< \frac{1}{M_0}.$
First we rewrite $Q$. Observe that 
\begin{eqnarray*}
 (u_i-u_1)(\bar x)&=&\sum_{j\geq 1} \, (u_{ij}-u_{1j})x_j=u_{i1}-u_{11} + \sum_{j\geq 2} (u_{ij}-u_{i1}+u_{11}-u_{1j})\, x_j.
\end{eqnarray*}
If we note the vector $N:=(u_{11}-u_{j1})_j$ and the matrix
$M:=(M_{ij})=u_{ij}-u_{i1}+u_{11}-u_{1j}.$
Therefore, 
\begin{equation}\label{cond-dec-3}
sign( Q(\bar x))= sign(-N\bar x+\bar{x}^tM \bar x).           
\end{equation}
So we have to find the region given by $\{\bar x: -N\bar x+\bar{x}^tM \bar x<0\}$. To deal with it, we apply Lemma \ref{gralA}
and we use equation (\ref{NM2}) and the theorem is concluded. \qed

\begin{obs}\label{a1-obs2} If we apply the proof of Lemma \ref{gralA} to the particular case in which $s$ is the only invader, we have that $sign(Q(x_2))= sign(x_2[u_{21}-u_{11}+  (u_{22}-u_{12}+u_{11}-u_{21}) \, x_2])$. Hence, $Q(s)=0$ if and only if $x_2=0$ or $ x_2 =\frac{u_{11}-u_{21}}{u_{11}-u_{21}+u_{22}-u_{12}}=\frac{1}{1+\frac{u_{22}-u_{12}}{u_{11}-u_{21}}}=p_{12}$ and so $Q(s)<0$, for any $0<s<p_{12}.$
In particular, if we apply this to Theorem \ref{gralA}, it follows that the whole segment $[0,p_{12})$ is in the basin of attraction of $e_1.$
\end{obs}

\subsubsection{Proof of Theorem \ref{ULBA for games}: applying Theorem \ref{a1}}
\label{newconditions}

\noi{\it Proof of theorem \ref{ULBA for games}.} The proof follows immediately from Theorem \ref{a1} and the definition of $M(s)$.
In fact, ordering  the strategies in such a way that $s$ corresponds to the first one and  $N(s, s_i)\geq N(s, s_j)$ if $j>i$ and to be coherent with notation, then 
\begin{eqnarray*}
 M_{ij}+M_{ji}&= &U_{\de,p}(s,s)-U_{\de,p}(s_i,s)+U_{\de,p}(s,s)-U_{\de,p}(s_j,s)\\
& +& U_{\de,p}(s_j,s_i)-U_{\de,p}(s,s_i)+U_{\de,p}(s_i,s_j)-U_{\de,p}(s,s_j).
\end{eqnarray*}
Given that $N(s, s_i)\geq N(s, s_j)$ for $j>i$,it follows that 
$$\frac{M_{ij}+M_{ji}}{N_{i}}\leq 2+\frac{ U_{\de,p}(s_j,s_i)-U_{\de,p}(s,s_i)+U_{\de,p}(s_i,s_j)-U_{\de,p}(s,s_j)}{U_{\de,p}(s,s)-U_{\de,p}(s_i,s)}$$
and so the constant $M_0=\sup\{\frac{M_{ij}+M_{ji}}{-N_{i}},\,  0\} \leq M(s)+2$. By Theorem \ref{a1}, $\De_{\frac{1}{M(s)+2}}=\{x: \sum_{i\geq 2}\,  x_i \leq \frac{1}{M(s)+2}\}\subset B_{loc}(e_1)$ and $s$ has a uniformly large basin of attraction given that $s$ satisfies the trifecta condition ($M(s)$ is bounded).  
\qed

\subsubsection{Preliminary results regarding the cross ratio condition}
\label{alcanza de a 2}
In this section we provide some preliminary results which will be used in some calculations in the next section and in Section \ref{ap for sect6}.

This section deals with the calculation of the cross ratio condition. This conditions requires bounding the following ratio: $\frac{U_{\de,p}(s,s)-U_{\de,p}(s,s')}{U_{\de,p}(s,s)-U_{\de,p}(s',s)}$. The main idea is that to calculate both the numerator and denominator, we can focus on the histories of first divergence between the two strategies. That is, when calculating the denominator we can focus on the finite histories such that the behavior of the pair of strategies $(s,s)$ without trembles deviates for the first time from that of $(s',s)$; and when calculating the numerator we can focus on the finite histories such that the behavior of the pair of strategies $(s,s)$ without trembles deviates for the first time from that of $(s,s')$. We also show that the histories of first deviation in both cases are related: the behavior of $(s,s)$ differs from the behavior of $(s',s)$ at $h_t$, if and only if $(s,s)$ differs from the behavior of $(s,s')$ at $\hat h_t$.

Let us consider two strategies $s$ and $s'$ and let 
$$\cR_{s,s'}:= \{h_t\in H: s(h_t)\neq s'(h_t) \,\mbox{and} \,s(h_\tau)=s'(h_\tau) \forall h_\tau \,\mbox{such that}\, h_t = h_\tau h_{t-\tau} \},$$
$$\cE_{s,s'}:= \{h_t\in H: s(h_t)=s'(h_t) \,\mbox{and} \,s(h_\tau)=s'(h_\tau) \forall h_\tau \,\mbox{such that}\, h_t = h_\tau h_{t-\tau} \}.$$

The set $\cR_{s,s'}$ includes all the finite histories for which the behavior of $s$ differs from the behavior of $s'$ for the first time. Observe that if $s\neq s'$ then $\cR_{s, s'}\neq \emptyset.$ The set $\cE_{s,s'}$ includes all the finite histories for which the behavior of $s$ does not differ from the behavior of $s'$ and has not differed in the past. If $s(h_0)\neq s'(h_0)$, then $\cE_{s,s'}=\emptyset.$

Recall that with $U_{\de, p}(s, s'/h_t)$ we denote the utility with seed $h_t$.

\begin{lem}\label{calc eq paths}
For any two strategies $s$ and $s'$, it follows that: 
\begin{eqnarray}\label{differences}
U_{\de,p}(s,s)-U_{\de,p}(s',s)= \sum_{ h_t\in \cR_{s,s'}}\,\de^{t} p_{s,s}(h_t)(\,U_{\de,p}(s,s/h_t)-U_{\de,p}(s',s/h_t)).
\end{eqnarray}
\end{lem}

\begin{proof}
The lemma follows from the fact that $p_{s,s}(h_t)=p_{s',s}(h_t)$ for any $h_t\in\cE_{s,s'}\cup\cR_{s,s'}$ and the definition of $U_{\de,p}(s,s/h_t).$
\end{proof}

Now we consider the set 
$$\widehat\cR_{s,s'}:= \{h_t\in H: s(\hat h_t)\neq s'(\hat h_t) \,\mbox{and} \,s(\hat h_\tau)=s'(\hat h_\tau) \forall h_\tau \,\mbox{such that}\, h_t = h_\tau h_{t-\tau} \}.$$

\begin{lem}\label{calc eq paths bis} Given two strategies $s$ and $s'$, it follows that: 
\begin{eqnarray*}\label{differences2}
U_{\de,p}(s,s)-U_{\de,p}(s,s')&=  & \sum_{ h_t\in \widehat\cR_{s,s'}}\,\de^{t} p_{s,s}( h_t)(\,U_{\de,p}(s,s/h_t)-U_{\de,p}(s,s'/ h_t)).
\end{eqnarray*}
\end{lem}

We omit the proof of Lemma \ref{calc eq paths bis} as it is similar to that of Lemma \ref{calc eq paths}.
 
\begin{lem}\label{biyeccion}Given two strategies $s$ and $s',$ the map $h_t\to \hat h_t$ is a bijection between $\cR_{s, s'}$ and $\hat\cR_{s, s'}$, such that, $h_t\in \cR_{s,s'}$ if and only if $\hat h_t\in \widehat\cR_{s,s'}.$ 

\end{lem}
\begin{proof} The only if direction of the lemma follows from the definition of $\cR_{s,s'}$ and $\widehat\cR_{s,s'}.$ The if direction then follows immediately from the fact that $\hat{\hat h}= h.$
\end{proof}

Lemma \ref{biyeccion} the histories of first deviation in both cases are related: if the behavior of $(s,s)$ differs from the behavior of $(s',s)$ at $h_t$, if and only if $(s,s)$ differs from the behavior of $(s,s')$ at $\hat h_t$. This allows us to calculate $U_{\de,p}(s,s)-U_{\de,p}(s,s')$ focusing on histories in $\cR_{s,s'}$, which will allow for simple calculations in some proofs as we would use the same set of histories in both the numerator and denominator of the cross ratio condition.

\begin{lem}\label{calc eq paths hat} Given two strategies $s$ and $s'$ it follows that 
 $$U_{\de,p}(s,s)-U_{\de,p}(s,s')=   \sum_{ h\in \cR_{s,s'}}\,\de^{t} p_{s,s}( h_t)(\,U_{\de,p}(s,s/\hat h_t)-U_{\de,p}(s,s'/\hat h_t)).$$
\end{lem}

\begin{proof}
The lemma follows from Lemmas \ref{calc eq paths bis} and \ref{biyeccion}.
\end{proof}

\subsubsection{Proof of Theorem \ref{2 implica ULBC}: the cross ratio condition implies the trifecta condition}
\label{two to ulbc}

We have to show that $s$ satisfies the trifecta condition to prove that $s$ has a uniformly large basin of attraction; in other words, we need to show that there exists $L'$ such that for any $s^*, s'$ such that $U_{\de,p}(s, s)-U_{\de,p}(s^*, s)\geq U_{\de,p}(s, s)-U_{\de,p}(s', s)$ it follows that $M_{\de,p}(s,s^*,s')< L'.$

Observe that 
\begin{eqnarray*}
& & M_{\de,p}(s,s^*,s')=\frac{U_{\de,p}(s',s^*)+U_{\de,p}(s^*,s')-U_{\de,p}(s, s)-U_{\de,p}(s, s),}{U_{\de,p}(s, s)-U_{\de,p}(s^*, s)}=\\
& & \frac{U_{\de,p}(s',s^*)+U_{\de,p}(s^*,s')-2U_{\de,p}(s,s)}{U_{\de,p}(s, s)-U_{\de,p}(s^*, s)}+ \frac{U_{\de,p}(s,s)-U_{\de,p}(s, s^*)}{U_{\de,p}(s, s)-U_{\de,p}(s^*, s)}\\
& & +\frac{U_{\de,p}(s,s)-U_{\de,p}(s, s')}{U_{\de,p}(s, s)-U_{\de,p}(s^*, s)}.
\end{eqnarray*}
The second term, $\frac{U_{\de,p}(s,s)-U_{\de,p}(s, s^*)}{U_{\de,p}(s, s)-U_{\de,p}(s^*, s)}$, is bounded by the the cross ratio condition.
The third term, $\frac{U_{\de,p}(s,s)-U_{\de,p}(s, s')}{U_{\de,p}(s, s)-U_{\de,p}(s^*, s)}$, is smaller than $\frac{U_{\de,p}(s,s)-U_{\de,p}(s, s')}{U_{\de,p}(s, s)-U_{\de,p}(s', s)}$ given that $U_{\de,p}(s, s)-U_{\de,p}(s^*, s)\geq U_{\de,p}(s, s)-U_{\de,p}(s', s)$. Hence, the third term is also bounded by the cross ratio condition. So, to finish, we have to bound the first term:
\begin{eqnarray*}
& & \frac{U_{\de,p}(s',s^*)+U_{\de,p}(s^*,s')-2U_{\de,p}(s,s)}{U_{\de,p}(s, s)-U_{\de,p}(s^*, s)}=\\
& & \frac{U_{\de,p}(s',s^*)-U_{\de,p}(s,s) +U_{\de,p}(s^*,s')-U_{\de,p}(s,s)}{U_{\de,p}(s, s)-U_{\de,p}(s^*, s)}.
 \end{eqnarray*}
The bound is going to follow from the fact that $s$ is an uniformly efficient perfect public equilibrium.

We start by calculating $U_{\de,p}(s',s^*)-U_{\de,p}(s,s)$ and $U_{\de,p}(s^*,s')-U_{\de,p}(s,s).$ As in the previous section we do this by focusing on the histories of first difference between strategies, with the difference that now we have to keep track of the behavior of three different strategies.

Let us consider three strategies $s,$ $s^*,$and $s'$ and let 
$$\widehat \cE_{s,s'}:= \{h_t\in H: s(\hat h_t)=s'(\hat h_t) \,\mbox{and} \,s(\hat h_\tau)=s'(\hat h_\tau) \forall h_\tau \,\mbox{such that}\, h_t = h_\tau h_{t-\tau} \},$$
$$C_{s,s^*,s'}:=\cR_{s,s^*} \cap ( \widehat \cE_{s,s'} \cup \widehat \cR_{s,s'} ),$$
$$\widehat C_{s,s^*,s'}:=\widehat \cR_{s,s'} \cap \cE_{s,s^*},$$
where $\cR_{s,s'},$ $\widehat \cR_{s,s'},$and $\cE_{s,s'}$ are as defined in the previous section.
Note that $C_{s,s^*,s'}$ is the set of histories in which $s^*$ differs for the first time from $s$ and $s'$ has not differed from $s$ before that history. $\widehat C_{s,s^*,s'}$ is the set of histories such that $s^*$ has not differed from $s$ before or in that history and $s'$ differs for the first time from $s$ in the ``mirror'' of that history ($\hat h_t$).
 
Arguing as in Lemmas \ref{calc eq paths} and \ref{calc eq paths hat}, it follows that:
\begin{eqnarray*}
 U_{\de,p}(s^*,s')-U_{\de,p}(s,s)= & & 
 \sum_{ h_t \in C_{s,s^*,s'}}\,\de^t p_{s,s}(h_t)(\,U_{\de,p}(s^*,s'/h_t)-U_{\de,p}(s,s/h_t)) \\
 & & + \sum_{ h_t \in \widehat C_{s,s^*,s'}}\,\de^t p_{s,s}(h_t)(\,U_{\de,p}(s^*,s'/h_t)-U_{\de,p}(s,s/h_t)).
\end{eqnarray*}
Similarly,
\begin{eqnarray*}
 U_{\de,p}(s',s^*)-U_{\de,p}(s,s)= & & 
 \sum_{ h_t \in C_{s,s^*,s'}}\,\de^t p_{s,s}(h_t)(\,U_{\de,p}(s',s^*/\hat h_t)-U_{\de,p}(s,s/\hat h_t)) \\
 & & + \sum_{ h_t \in \widehat C_{s,s^*,s'}}\,\de^t p_{s,s}(h_t)(\,U_{\de,p}(s',s^*/\hat h_t)-U_{\de,p}(s,s/\hat h_t)).
\end{eqnarray*}
Therefore, we get that
\begin{eqnarray} \label{numeradordecomposed}
& & U_{\de,p}(s^*,s')+U_{\de,p}(s',s^*)-U_{\de,p}(s,s)-U_{\de,p}(s,s)=  
\end{eqnarray}
\begin{eqnarray*} 
& & \sum_{ h_t \in C_{s,s^*,s'}}\,\de^t p_{s,s}(h_t)(\,U_{\de,p}(s^*,s'/h_t)+U_{\de,p}(s',s^*/\hat h_t)-U_{\de,p}(s,s/h_t)-U_{\de,p}(s,s/\hat h_{\hat t^*_h})) \\
 & & +\sum_{ h_t \in \widehat C_{s,s^*,s'}}\,\de^t p_{s,s}(h_t)(\,U_{\de,p}(s^*,s'/h_t)+U_{\de,p}(s',s^*/\hat h_t)-U_{\de,p}(s,s/h_t)-U_{\de,p}(s,s/\hat h_t)).
\end{eqnarray*}
 Since $T+S < 2R$, $U_{\de,p}(s^*,s'/h_t)+U_{\de,p}(s',s^*/\hat h_t)\leq 2R$, and it follows that
\begin{eqnarray*}
& &  U_{\de,p}(s^*,s'/h_t)+U_{\de,p}(s',s^*/\hat h_t)-U_{\de,p}(s,s/h_t)-U_{\de,p}(s,s/\hat h_t) <\\
& &< 2R- U_{\de,p}(s,s/h_t)-U_{\de,p}(s,s/\hat h_t) .
\end{eqnarray*}
Since $(s,s)$ is uniformly efficient, it follows that $R- U_{\de,p}(s,s/h_t)< C_1(1-\de p^2)$. Hence, 
\begin{eqnarray}\label{t*}U_{\de,p}(s^*,s'/h_t)+U_{\de,p}(s',s^*/\hat h_t)-U_{\de,p}(s,s/h_t)-U_{\de,p}(s,s/\hat h_t)  <2C_1(1-\de p^2).
\end{eqnarray}
Now, recalling that $U_{\de,p}(s, s)-U_{\de,p}(s^*, s)$ is the sum of all payoff's difference starting at all first deviation histories and since $s$ is a perfect public equilibrium then $U_{\de,p}(s,s/h_t)-U_{\de,p}(s^*,s/h_t)\geq0$ for any $h_t$, it follows that 
\begin{eqnarray}\label{ss*}U_{\de,p}(s, s)-U_{\de,p}(s^*, s) \geq \sum_{ h_t \in C_{s,s^*,s'}} \de^t p_{s,s}(h_t)(\,U_{\de,p}(s,s/h_t)-U_{\de,p}(s^*,s/h_t)).
 \end{eqnarray}
Similarly,
$$ U_{\de,p}(s, s)-U_{\de,p}(s', s) \geq \sum_{ h_t \in \widehat C_{s,s^*,s'} }\,\de^{t} p_{s,s}(h_t)(\,U_{\de,p}(s,s/ \hat h_t)-U_{\de,p}(s',s/ \hat h_t)).$$

Hence, from the fact that $U_{\de,p}(s, s)-U_{\de,p}(s^*, s)\geq U_{\de,p}(s, s)-U_{\de,p}(s', s)$, it follows that 
\begin{eqnarray}\label{ss'}
 U_{\de,p}(s, s)-U_{\de,p}(s^*, s)\geq \sum_{ h_t \in \widehat C_{s,s^*,s'}}\,\de^t p_{s,s}(h_t)(\,U_{\de,p}(s,s/\hat h_t)-U_{\de,p}(s',s/\hat h_t)).
\end{eqnarray} 

By inequality (\ref{numeradordecomposed}) and the fact that $U_{\de,p}(s, s)-U_{\de,p}(s^*, s)\geq U_{\de,p}(s, s)-U_{\de,p}(s', s)$, it follows that:
\begin{eqnarray*}
 & & \frac{U_{\de,p}(s',s^*)+U_{\de,p}(s^*,s')-U_{\de,p}(s, s)-U_{\de,p}(s, s)}{U_{\de,p}(s, s)-U_{\de,p}(s^*, s)}\leq\\
& & \frac{\sum_{ h_t \in C_{s,s^*,s'}}\de^t p_{s,s}(h_t)(U_{\de,p}(s^*,s'/h_t)+U_{\de,p}(s',s^*/\hat h_t)-U_{\de,p}(s,s/h_t)-U_{\de,p}(s,s/\hat h_t))}{U_{\de,p}(s, s)-U_{\de,p}(s^*, s)} \\
 & & + \frac{\sum_{ h_t \in \widehat C_{s,s^*,s'}}\de^t p_{s,s}(h_t)(U_{\de,p}(s^*,s'/h_t)+U_{\de,p}(s',s^*/\hat h_t)-U_{\de,p}(s,s/h_t)-U_{\de,p}(s,s/\hat h_t))}{U_{\de,p}(s, s)-U_{\de,p}(s', s)}
\end{eqnarray*}
Then, by inequalities (\ref{ss*}) and (\ref{ss'}) it follows that:
\begin{eqnarray*}
 & & \frac{U_{\de,p}(s',s^*)+U_{\de,p}(s^*,s')-U_{\de,p}(s, s)-U_{\de,p}(s, s)}{U_{\de,p}(s, s)-U_{\de,p}(s^*, s)}\leq\\
& & \frac{\sum_{ h_t \in C_{s,s^*,s'}}\de^t p_{s,s}(h_t)(U_{\de,p}(s^*,s'/h_t)+U_{\de,p}(s',s^*/\hat h_t)-U_{\de,p}(s,s/h_t)-U_{\de,p}(s,s/\hat h_t))}{ \sum_{ h_t \in C_{s,s^*,s'}} \de^t p_{s,s}(h_t)(U_{\de,p}(s,s/h_t)-U_{\de,p}(s^*,s/h_t))   } \\
 & & + \frac{\sum_{ h_t \in \widehat C_{s,s^*,s'}}\de^t p_{s,s}(h_t)(U_{\de,p}(s^*,s'/h_t)+U_{\de,p}(s',s^*/\hat h_t)-U_{\de,p}(s,s/h_t)-U_{\de,p}(s,s/\hat h_t))}{ \sum_{ h_t \in \widehat C_{s,s^*,s'}}\de^t p_{s,s}(h_t)(U_{\de,p}(s,s/\hat h_t)-U_{\de,p}(s',s/\hat h_t))}
\end{eqnarray*}

From inequality (\ref{t*}), and the fact that $s$ is an uniformly strict perfect equilibrium, it follows that the last two terms in previous set of inequalities are bounded by  
\begin{eqnarray*}
& & \leq \frac{\sum_{ h_t \in C_{s,s^*,s'}}\,\de^t p_{s,s}(h_t)2C_1(1-\de p^2)}{ \sum_{ h_t \in C_{s,s^*,s'}} \de^t p_{s,s}(h_t)C_0(1-\de p^2)  }  + \frac{\sum_{ h_t \in \widehat C_{s,s^*,s'}}\,\de^t p_{s,s}(h_t)2C_1(1-\de p^2)}{ \sum_{ h_t \in \widehat C_{s,s^*,s'}}\,\de^t p_{s,s}(h_t)C_0(1-\de p^2)}\\
& & \leq \frac{2C_1}{C_0}\frac{\sum_{ h_t \in C_{s,s^*,s'}}\,\de^t p_{s,s}(h_t)}{ \sum_{ h_t \in C_{s,s^*,s'}} \de^t p_{s,s}(h_t)  }  + \frac{2C_1}{C_0}\frac{\sum_{ h_t \in \widehat C_{s,s^*,s'}}\,\de^t p_{s,s}(h_t)}{ \sum_{ h_t \in \widehat C_{s,s^*,s'}}\,\de^t p_{s,s}(h_t)} =4\frac{C_1}{C_0},
\end{eqnarray*}
where the last equality follows observing that in each term of the sum the probability factor in the numerator and denominator are the same.
\qed

\subsection{Proofs for Section \ref{W banca tr subs}: strategies with uniformly large basins of attraction}\label{ap for sect6}

\subsubsection{Approximating payoffs for small probability of mistakes}
\label{calculating}

In this section, we show how to bound the cross ratio condition by focusing on 0-tremble histories (i.e. the histories that would be reached if $p=1$). This greatly simplifies verifying that the cross ratio conditions holds.   

From now on with $U_{\de, p }(h_{s,s'/h_t})$ we denote the discounted sum of utilities in $h_{s,s'/h_t}.$ With $U_{\de, p}( h^c_{s,s'/h_t})$  we denote $U_{\de, p}(s,s'/h_t)-U_{\de, p }(h_{s,s'/h_t})$, that is the discounted sum of utilities in histories outside of $h_{s,s'/h_t}.$ Also, with $\nee(h_t)$ we denote the  set of histories which are not in $h_{s,s'/h_t}$; those histories are called second order histories given $h_t$.

To simplify calculations we change the usual renormalization factor $1-\de$ by $\frac{1-p^2\de}{p^2}$ and we calculate the payoff as follows:
 $U_{\de,p}(s, s')= \frac{1-p^2\de}{p^2}\sum_{t\geq 0, a_t,b_t }   \de^t   p_{s,s}(a_t,b_t)u(a^t,b^t).$ 
Both ways calculating the payoff (either with renormalization $1-\de$ or  $\frac{1-p^2\de}{p^2}$) are equivalent as they rank histories in the same way.

In what follows, we restrict the probabilities of mistakes in relation to the discount factor such that $p>p(\de)$ where 
 \begin{eqnarray}\label{p}
p(\de)=max\Bigg\{ \sqrt{1-\frac{1}{16}\frac{C_0}{G}(1-\de)^2}, \sqrt{\de}, \sqrt{1+\frac{1}{5}(\de-1)}\Bigg\}
\end{eqnarray}
for some positive constant $G$ that depends on the payoff matrix of the stage game and the positive constant $C_0$ from Definition \ref{sgpstrict}. This restriction says that the probability of mistakes is much smaller than the impatience and it allows us to focus on utilities of 0-tremble histories. Moreover, we assume that $\de>\frac{1}{2}$.

\begin{teo}\label{estimating ***} If $(s,s)$ is an uniformly strict perfect public equilibrium and there exists $C_4$ and $\de_0$ such that
\begin{eqnarray}\label{main bound}
 \frac{U_{\de,p}(h_{s,s/ \hat h_t})-U_{\de,p}(h_{s,s'/\hat h_t})}{U_{\de,p}(h_{s,s/h_t})-U_{\de,p}(h_{s',s/h_t})}< C_4,
\end{eqnarray}
for any $h_t$, $s'$, $\de>\de_0$ and $p>p(\de)$, then $s$ satisfies the cross ratio condition.
\end{teo}

This theorem says that it is enough to focus on payoffs of 0-tremble histories to calculate that the cross ratio condition is satisfied. For the proof of this theorem, we first develop a series of lemmas.

\begin{lem}\label{comp1} For any pair of strategies $s, s'$ it follows that  
 $|U_{\de, p}( h^c_{s',s/h_t})|< \frac{1-p^2}{p^2 (1-\de)} G$ where $G=\max\{|T|, |S|\}.$  
\end{lem}

\begin{proof} Observe that, for a fixed $\tau>t$,   $\sum_{h_\tau \in H_\tau }  \,\,   p_{s,s'/h_t}(h_\tau)=1.$ Since in the $0-$tremble path at time $t$ the probability is $p^{2t+2}$, it follows that  $\sum_{h_\tau \in H_\tau \cap \nee(h_t)}\,\,     p_{s,s'/h_t}(h_\tau)=1-p^{2t+2}.$
Therefore, and recalling that $u(a^t,b^t)\leq G$ for any $(a^t,b^t)$ and that $(a_t,b_t)$ is a history including the outcome in period $t$:
\begin{eqnarray*}
& & |U_{\de, p}( h^c_{s',s/h_t})|= | \frac{1-p^2\de}{p^2}\sum_{\tau \geq t,(a_\tau,b_\tau) \in \nee(h_t)}\,\,  \de^{\tau-t}   p_{s,s' / h_t}(a_\tau,b_\tau)u(a^\tau,b^\tau)|\\
&\leq &\frac{1-p^2\de}{p^2}\sum_{\tau \geq t}\de^{\tau-t}\sum_{(a_\tau,b_\tau) \in \nee(h_t)} p_{s,s' / h_t}(a_\tau,b_\tau)|u(a^\tau,b^\tau)|\\
&\leq& \frac{1-p^2\de}{p^2}G \sum_{\tau \geq t}\de^{\tau-t} (1-p^{2(\tau-t)+2})\\
&=&\frac{1-p^2\de}{p^2}G[\sum_{\tau \geq t}\de^{\tau-t} - \sum_{\tau \geq t}\de^{\tau-t} p^{2(\tau-t)+2}]=\frac{1-p^2\de}{p^2}G[ \frac{1}{1-\de}- \frac{p^2}{1-p^2\de}]\\
&=&\frac{1-p^2}{p^2 (1-\de)}G.
\end{eqnarray*} \end{proof}

From now on, we denote
\begin{eqnarray*}
N_{\de,p}(s,s'):= U_{\de,p}(s,s)-U_{\de,p}(s',s),\\
\bar N_{\de,p}(s,s'):= U_{\de,p}(s,s)-U_{\de,p}(s,s'),
\end{eqnarray*}
and recalling equality (\ref{differences}) we define 
\begin{eqnarray*}
N^e_{\de,p}(s,s'):= \sum_{h_t\in \cR_{s,s'}}\,\de^{t} p_{s,s}(h_t)[\,U_{\de,p}(h_{s,s/h_t})-U_{\de,p}(h_{s',s/h_t})],\\
\bar N^e_{\de,p}(s,s'):= \sum_{h_t\in \widehat\cR_{s,s'}}\,\de^{t} p_{s,s}(h_t)[\,U_{\de,p}(h_{s,s/ h_t})-U_{\de,p}(h_{s,s'/h_t})].
\end{eqnarray*}

\begin{lem}\label{bounding n bar n b}
For any two strategies $s$ and $s'$, it holds that: 
\begin{enumerate} 
\item $N_{\de,p}(s,s')\geq N^e_{\de,p}(s,s')-2\frac{1-p^2}{p^2 (1-\de)} G;$
\item $N^e_{\de,p}(s,s')\geq N_{\de,p}(s,s')-2\frac{1-p^2}{p^2 (1-\de)} G;$
\item $\bar N_{\de,p}(s,s')\leq \bar N^e_{\de,p}(s,s')+2\frac{1-p^2}{p^2 (1-\de)} G;$
\end{enumerate}
\end{lem}

\begin{proof}
Directly by Lemma \ref{comp1}, the definitions of $N_{\de,p}(s,s'),$ $N^e_{\de,p}(s,s'),$ $\bar N_{\de,p}(s,s'),$ and $U_{\de,p}(h_{s,s'})$, and the fact that $\sum_{h_t\in \cR_{s,s'}}\,\de^{t} p_{s,s}(h_t)<1$ and $\sum_{h_t\in \widehat \cR_{s,s'}}\,\de^{t} p_{s,s}(h_t)<1.$
\end{proof}

\begin{lem}\label{lem-less1}
If $p>p(\de)$ and $\de>\frac{1}{2}$, then $\frac{1-\de}{(1-p^2\de)p^2}<2$. 
\end{lem}

\begin{proof}
Given that $\sqrt{\de}<p$, by condition (\ref{p}), and $\de>\frac{1}{2}$, it follows that $\frac{1}{p^2}<2$. Given that $p>1$, it holds that $\frac{1-\de}{(1-p^2\de)}<1$. The lemma follows from combining these two results. 
\end{proof}

\begin{lem}\label{lem-less2}
If $p>p(\de)$ and $\de>\frac{1}{2}$, then $2\frac{G(1-p^2)}{C_0(1-\de)(1-p^2\de)p^2}<\frac{1}{4}$. 
\end{lem}

\begin{proof}
By condition (\ref{p}), $p>\sqrt{1-\frac{C_0(1-\de)^2}{16G}}$. Rearranging, we have that $\frac{16G(1-p^2)}{C_0(1-\de)(1-p^2\de)}<1$, and using Lemma \ref{lem-less1}, $\frac{16G(1-p^2)}{C_0(1-\de)(1-p^2\de)p^2}<2$. The results follows from dividing both sides by 8.
\end{proof}

\begin{lem}\label{bounding ppe} If $(s,s)$ is a uniformly strict perfect public equilibrium with constant $C_0$, it holds that $N^e_{\de,p}(s,s')> \frac{3}{4}C_0(1-p^2\de)$, for every other strategy $s'$ and every history $h_t$ such that $s(h_t)\neq s'(h_t)$ when $p>p(\de)$ and $\de$ is sufficiently large.
\end{lem}

\begin{proof}
By Lemma \ref{bounding n bar n b}\textit{(ii)}, $N^e_{\de,p}(s,s')\geq N_{\de,p}(s,s')-2\frac{1-p^2}{p^2 (1-\de)} G$. Taking $N_{\de,p}$ as a common factor we have that $N^e_{\de,p}(s,s')\geq N_{\de,p}(s,s')[1-2\frac{(1-p^2)G}{p^2 (1-\de)N_{\de,p}(s,s')}]$. By $(s,s)$ being a uniformly strict perfect public equilibrium, we have that $N_{\de,p}(s,s')>C_0(1-p^2\de)$, and hence $N^e_{\de,p}(s,s')> C_0(1-p^2\de)[1-2\frac{(1-p^2)G}{p^2 (1-\de)C_0(1-p^2\de)}]$. Hence, by Lemma \ref{lem-less2}, $N^e_{\de,p}(s,s')> \frac{3}{4}C_0(1-p^2\de)$.
\end{proof}

Based on the previous lemma, the cross ration condition can be bounded by a similar condition which focuses on payoffs in the 0-tremble histories.

 \begin{lem}\label{bounding n bar n b 2}If $(s,s)$ is a uniformly strict perfect public equilibrium, then $\frac{\bar N_{\de,p}(s,s')}{N_{\de,p}(s,s')}\leq \frac{3}{2}\frac{\bar N^e_{\de,p}(s,s')}{N^e_{\de,p}(s,s')}+\frac{1}{2}$ for $p>p(\de)$ and sufficiently high $\de$. 
 \end{lem}

\begin{proof} From Lemma \ref{bounding n bar n b} \textit{(i)} and \textit{(iii)}, and taking $N^e_{\de p}(s,s')$ as common factor, it follows that $\frac{\bar N_{\de,p}(s,s')}{N_{\de,p}(s,s')} \leq \frac{\bar N^e_{\de,p}(s,s')+ 2\frac{1-p^2}{p^2 (1-\de)} G}{N^e_{\de,p}(s,s')(1-2\frac{1-p^2}{p^2 (1-\de)} G \frac{1}{N^e_{\de p}(s,s')})}.$ Then, by $(s,s)$ being a uniformly strict perfect public equilibrium and Lemma \ref{bounding ppe}, it follows that
$\frac{\bar N_{\de,p}(s,s')}{N_{\de,p}(s,s')} < \frac{\bar N^e_{\de,p}(s,s')+ 2\frac{1-p^2}{p^2 (1-\de)} G}{N^e_{\de,p}(s,s')(1-2G\frac{1-p^2}{(1-\de)p^2 \frac{3}{4}C_0 (1-p^2\de)})}.$ Hence, by Lemma \ref{lem-less2} and Lemma \ref{bounding ppe} again, it follows that
$\frac{\bar N_{\de,p}(s,s')}{N_{\de,p}(s,s')} < \frac{3}{2}\frac{\bar N^e_{\de,p}(s,s')}{N^e_{\de,p}(s,s')}+\frac{1}{2}.$
\end{proof}

{\em Proof of Theorem \ref{estimating ***}:} From Lemma \ref{bounding n bar n b 2}, we need to bound 
$$\frac{\bar N^e_{\de,p}(s,s')}{N^e_{\de,p}(s,s')} = \frac{\sum_{h_t\in \widehat\cR_{s,s'}}\,\de^{t} p_{s,s}(h_t)[\,U_{\de,p}(h_{s,s/ h_t})-U_{\de,p}(h_{s,s'/h_t})]}{\sum_{h_t\in \cR_{s,s'}}\,\de^{t} p_{s,s}(h_t)[\,U_{\de,p}(h_{s,s/h_t})-U_{\de,p}(h_{s',s/h_t})]}.$$

Noting that $h_t \in \widehat\cR_{s,s'}$ if and only if $\hat h_t \in \cR_{s,s'}$, we can modified the summation in the numerator so that the summation set coincides with the one in the denominator:
\begin{eqnarray*}
& & \frac{\bar N^e_{\de,p}(s,s')}{N^e_{\de,p}(s,s')}=\frac{\sum_{h_t\in \cR_{s,s'}}\de^tp_{s,s}(h_t)(U_{\de, p}(h_{s,s/\hat h_t})-U_{\de, p}(h_{s,s'/\hat h_t}))}{\sum_{h_t\in \cR_{s,s'}}\de^tp_{s,s}(h_t)(U_{\de, p}(h_{s,s/ h_t})-U_{ \de, p}(h_{s',s/ h_t})) }\\
& & = \frac{\sum_{h_t \in \cR_{s,s'}}\de^t p_{s,s}(h_t)\frac{U_{\de, p}(h_{s,s/\hat h_t})-U_{\de, p}(h_{s,s'/\hat h_t})}{U_{ \de, p}(h_{s,s/ h_t})-U_{ \de, p}(h_{s',s/ h_t})}(U_{ \de, p}(h_{s,s/ h_t})-U_{\de, p}(h_{s',s/ h_t}))}{\sum_{h_t \in \cR_{s,s'}}\de^t p_{s,s}(h_t)(U_{ \de, p}(h_{s,s/ h_t})-U_{ \de, p}(h_{s',s/ h_t}))}.
\end{eqnarray*}
Hence, by condition (\ref{main bound}) in the statement of the theorem, it follows that
$$\frac{\bar N^e_{\de,p}(s,s')}{N^e_{\de,p}(s,s')}<C_4 \frac{\sum_{h_t\in \cR_{s,s'}}\de^tp_{s,s}(h_t)(U_{\de, p}(h_{s,s/h_t})-U_{\de, p}(h_{s',s/h_k}))}{\sum_{h_t in \cR_{s,s'}}\de^tp_{s,s}(h_t)(U_{\de, p}(h_{s,s/h_t})-U_{\de, p}(h_{s',s/h_t}))}=C_4.$$
\qed

Before being able to prove Theorem \ref{* type ULBA} we need to deal with two additional approximation of payoffs. First, we prove that the sum of discounted payoffs on the 0-tremble path for a uniformly efficient strategy playing against itself is close to $R$. As an intermediate step, we show next that this is also the case for the sum of discounted payoffs over all paths.

\begin{lem}\label{modeff} If $(s,s)$ is an uniformly efficient strategy, then for any $h_t$ it follows that $|R-U_{\de,p}(s,s/h_t)|< C_1(1-p^2\de)$ for $p>p(\de)$ and sufficiently high $\de$.
\end{lem}

\begin{proof} Note that for any pair of strategies $(s,s')$, the vector of discounted sum of payoffs $(U_{\de,p}(s,s'/h_t),U_{\de,p}(s',s/\hat h_t))$ must belong to the convex combination of the four possible payoff realizations: $(R,R)$, $(P,P)$, $(T,S)$, and $(S,T)$. The fact that $s$ is a uniformly efficient strategy implies that
$R-U_{\de,p}(s,s/\hat h_t)< C_1(1-p^2\de)$. Hence, given the previous comment about the set of possible payoffs, it follows that $U_{\de,p}(s,s/h_t)-R< C_1(1-p^2\de)\frac{T-R}{R-S}.$ Given that $2R>T+S$, it follows that $U_{\de,p}(s,s/h_t)-R< C_1(1-p^2\de)$ and the lemma follows given that $R-U_{\de,p}(s,s/h_t)< C_1(1-p^2\de)$ by the definition of a uniformly efficient strategy.   
\end{proof}

In other words, while the definition of an uniformly efficient strategy says that the sum of discounted payoffs cannot be too low relative to R, this also implies that they cannot be too high. The reason is that the possible vectors of discounted payoffs are limited to the convex combination of the possible four realizations of payoffs. As can be seen in Figure \ref{fig:eff}, if $U_{\de,p}(s,s/h_t)$ is not much lower than $R$, then $U_{\de,p}(s,s/\hat h_t)$ cannot be much higher than $R.$ 

\begin{figure}[!htbpp]
\begin{center}
\includegraphics[width=1\textwidth]{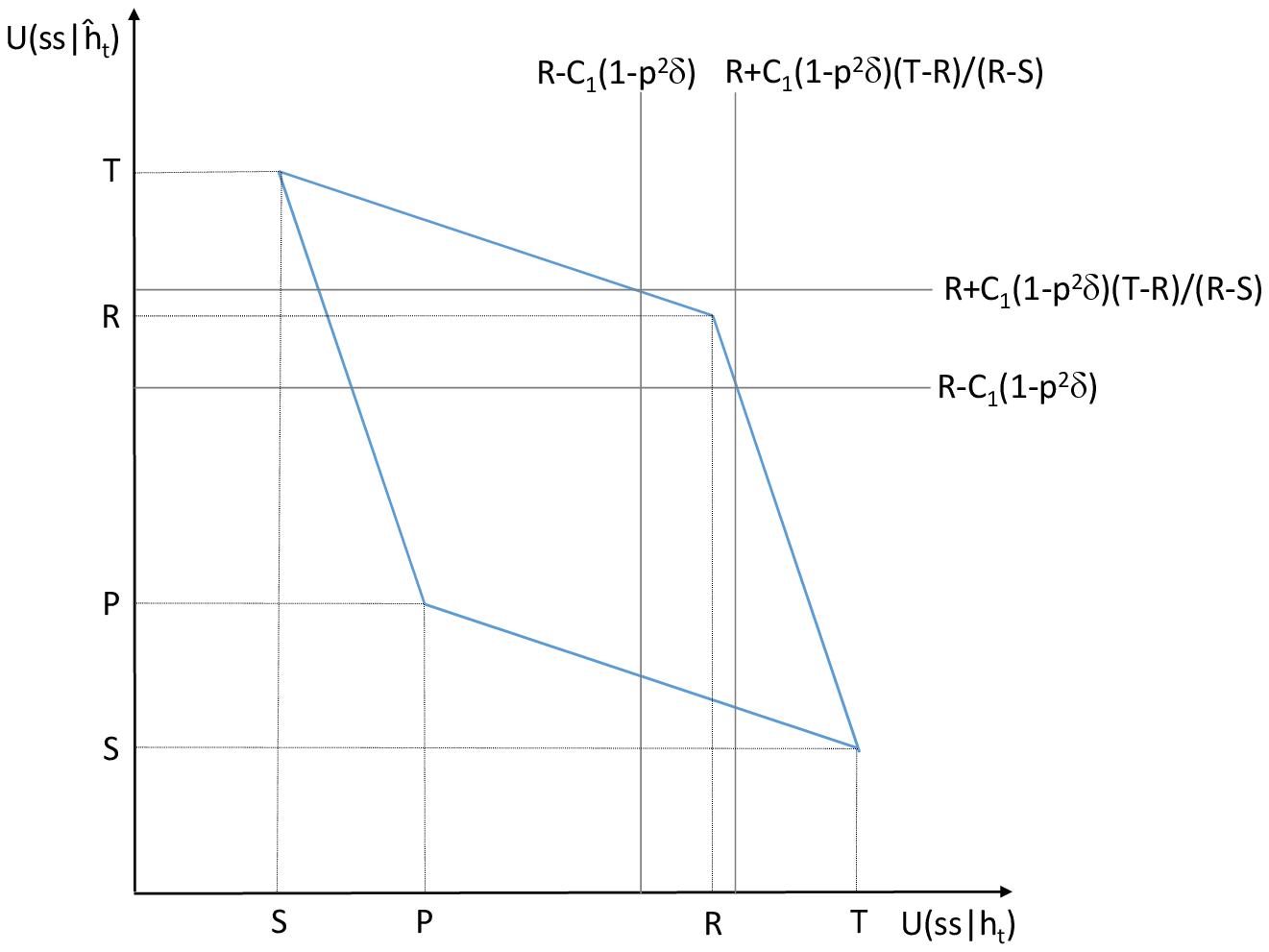}
 \end{center}
\caption{Uniform efficiency and possible payoffs}
\label{fig:eff}
\end{figure}

\begin{lem}\label{fr of C} If $(s,s)$ is an uniformly efficient strategy, then for any $h_t$ it follows that $|R-U_{\de,p}(h_{s,s/h_t})|< (C_1+G)(1-p^2\de)$ for $p>p(\de)$ and sufficiently high $\de.$ 
\end{lem}

\begin{proof} From the fact that $s$ is a uniformly efficient strategy and Lemma \ref{modeff}, $|R-U_{\de,p}(s,s/h_t)|< C_1(1-p^2\de).$ Hence, by $U_{\de,p}(s,s/h_t)=U_{\de,p}(h_{s,s/h_t})+U_{\de,p}(h^c_{s,s/h_t})$ and Lemma \ref{comp1}, it follows that 
$|R-U_{\de,p}(h_{s,s/h_t})|< C_1(1-p^2\de)+G\frac{1-p^2}{p^2 (1-\de)}.$
Given condition (\ref{p}), it can be shown that $\frac{1-p^2}{p^2 (1-\de)}<1-p^2\de$, and it follows that $|R-U_{\de,p}(h_{s,s/h_t})|<(C_1+G)(1-p^2\de).$
\end{proof}

Before moving to the proof of Theorem \ref{* type ULBA}, we deal with one last approximation of payoffs. We show that the loss in payoffs in the 0-tremble path from choosing another strategy is uniformly bounded away from zero if the strategy is a uniformly strict perfect public equilibrium.

\begin{lem}\label{aprox usppe} If $(s,s)$ is a uniformly strict perfect public equilibrium, then $U_{\de,p}(h_{s,s/h_t})-U_{\de,p}(h_{s',s/h_t}) > \frac{3}{4}C_0(1-p^2\de)$ for any $h_t$ and $s'$ such that $s(h_t)\neq s'(h_t)$ when $p>p(\de)$ and $\de$ is sufficiently high.
\end{lem}

\begin{proof} 
Given that $U_{\de,p}(s,s/h_t)=U_{\de,p}(h_{s,s/h_t})+U_{\de,p}(h^c_{s,s/h_t})$ and Lemma \ref{comp1}, it follows that $U_{\de,p}(h_{s,s/h_t})-U_{\de,p}(h_{s',s/h_t}) \geq U_{\de,p}(s,s/h_t)-U_{\de,p}(s',s/h_t)-2\frac{1-p^2}{p^2 (1-\de)}G$. The right hand side of this inequality can be written as $U_{\de,p}(s,s/h_t)-U_{\de,p}(s',s/h_t)[1-2\frac{(1-p^2)G}{p^2 (1-\de)(U_{\de,p}(s,s/h_t)-U_{\de,p}(s',s/h_t))}.$ By $(s,s)$ being a uniformly strict perfect public equilibrium, we have that $U_{\de,p}(s,s/h_t)-U_{\de,p}(s',s/h_t)>C_0(1-p^2\de)$ for any $h_t$ and $s'$ such that $s(h_t)\neq s'(h_t).$ Hence, $U_{\de,p}(h_{s,s/h_t})-U_{\de,p}(h_{s',s/h_t})> C_0(1-p^2\de)[1-2\frac{(1-p^2)G}{p^2 (1-\de)C_0(1-p^2\de)}]$. And so, by Lemma \ref{lem-less2}, $U_{\de,p}(h_{s,s/h_t})-U_{\de,p}(h_{s',s/h_t})> \frac{3}{4}C_0(1-p^2\de)$.
\end{proof}

\subsubsection{Proof of Theorem \ref{* type ULBA}: A star-type strategy has a uniformly large basin of attraction}

As a start-type strategy is uniformly efficient perfect public equilibrium, by Theorem \ref{2 implica ULBC}, it is enough that we prove that a star-type strategy satisfies the cross ration condition. And by Theorem \ref{estimating ***}, it is enough to show that inequality (\ref{main bound}) holds; that is, there exists $C_4$ such that
\begin{eqnarray}
 \frac{U_{\de,p}(h_{s,s/ \hat h_t})-U_{\de,p}(h_{s,s'/\hat h_t})}{U_{\de,p}(h_{s,s/h_t})-U_{\de,p}(h_{s',s/h_t})}< C_4,
\end{eqnarray}
for any $h_t$, $s'$, $\de>\de_0$ and $p>p(\de)$.

So, given $s'$ and $h_t$ we calculate $\frac{U_{\de, p}(h_{s,s/\hat h_t})-U_{\de,p}( h_{s,s'/\hat h_t})}{U_{\de, p}(h_{s,s/ h_t})-U_{\de,p}( h_{s',s/\hat h_t})}$
 where $h_{s,s/h_t}$ is the $0-$tremble path for $s,s$ starting with $h_t$
and $h_{s,s'/h_t}$ is the $0-$tremble path for $s,s'$ starting with $h_t$ and idem with $\hat h_t.$
Let us denote 
$$r= U_{\de, p}(h_{s,s/ h_t})-U_{\de,p}( h_{s',s/ h_t}),\,\,\,\hat r= U_{\de, p}(h_{s,s/ \hat h_t})-U_{\de,p}( h_{s,s'/ \hat h_t}),$$ 
$$\alpha= \frac{ U_{\de, p}(h_{s,s/ h_t})}{R},\,\,\, \hat \alpha= \frac{ U_{\de, p}(h_{s,s/ \hat h_t})}{R}.$$
 Let $b_R, b_S, b_T, b_P$ be the quantities given in Definition \ref{frequencies}, and using that $b_R+b_S+b_T+b_P=1$ it follows that 
\begin{eqnarray*} 
 r&=&\alpha R- (b_R R + b_S S + b_T T + b_P P)\\
& =&  \alpha R- R + R- (b_R R + b_S S + b_T T + b_P P)\\
&=&(\alpha-1)R+b_S(R-S)+b_T(R-T)+b_P(R-P).
\end{eqnarray*}
Then, given that $\gamma_2=\frac{R-S}{T-R}$, $\gamma_4=\frac{R-P}{T-R}$ and that $b_S\gamma_2+b_P\gamma_4-b_T>-C_5(1-p^2\de)$ by condition (\ref{eq on b}) in the definition of star-type strategies, it follows that  
\begin{eqnarray*}
r&=&(\alpha-1)R+b_S(T-R)\frac{(R-S)}{T-R}+b_T(R-T)+b_P(T-R)\frac{(R-P)}{T-R}\\
& =& (\alpha-1)R+ b_S (T-R) \gamma_2+ b_P (T-R) \gamma_4+ b_T(R-T)\\
& =&  (\alpha-1)R + (b_S\gamma_2+b_P\gamma_4-b_T)(T-R) +b_TR\\
&> & (\alpha-1)R +  b_TR -C_5(1-p^2\de)(T-R).
\end{eqnarray*}
So, 
\begin{eqnarray*}
b_T R &<&  r+ C_5(1-p^2\de)(T-R)-(\alpha-1)R \leq  r+ C_5(1-p^2\de)(T-R)+|\alpha-1|R\\
& =&  r(1+\frac{C_5(1-p^2\de)(T-R)}{r}+\frac{|\alpha-1|}{r} R).
\end{eqnarray*} 
Since $r> \frac{3}{4}C_0(1-p^2\de)$, from $(s,s)$ being a uniformly strict perfect public equilibrium and Lemma \ref{aprox usppe}, and $|\alpha-1|< (C_1+G)(1-p^2\de)$, by Lemma \ref{fr of C}, it follows that
\begin{eqnarray}\label{bound r}
b_TR< r(1+\frac{4C_5}{3C_0} (T-R)+\frac{4(C_1+G)}{3C_0} R).
\end{eqnarray}
Next we calculate $\hat r$ and compare it with $r.$ Following the same steps as for the calculation of $r$, and recalling that $h_{s,s'/\hat h_t}=\widehat{ h_{s',s/h_t}}$, it follows that
$$ \hat r= (\hat \alpha-1)R + b_S(R-T)+b_T( R-S)+b_P(R-P).$$
Then, given that $b_P(R-P)=r-(\alpha-1)R+b_S(R-S)+b_T(R-T)$, it follows that
\begin{eqnarray*}
\hat r& =&  (\hat \alpha-1)R+ b_S(R-T)+b_T(R-S)+b_P(R-P)\\
& =&  (\hat \alpha-\alpha)R+b_S(R-T)+b_T(R-S)-[b_S(R-S)+b_T(R-T)] +r \\
&=&  (\hat \alpha-\alpha)R+ r + (b_T-b_S)(T-S)\\
&\leq&  (\hat \alpha-\alpha)R+ r+ b_T(T-S)= (\hat \alpha-\alpha)R+ r + b_T R \frac{T-S}{R}.
\end{eqnarray*}

Then, given that $\hat \alpha-\alpha< |\hat \alpha-1|+|1-\alpha|$ and inequality (\ref{bound r}), it follows that
\begin{eqnarray*}
\hat r &\leq &  (|\hat \alpha-1|+|1-\alpha|)R+ r+ r (1+\frac{4C_5}{3C_0} (T-R) +\frac{4(C_1+G)}{3C_0} R)\frac{T-S}{R}\\
& =&  r(\frac{(|\hat \alpha-1|+|1-\alpha|)R}{r}+1+(1+\frac{4C_5}{3C_0} (T-R) +\frac{4(C_1+G)}{3C_0} R)\frac{T-S}{R}).
\end{eqnarray*}

From the fact that $s$ is a uniformly efficient strategy, it follows that $|1-\hat \alpha|R< (C_1+G)(1-\de p^2), |1-\alpha|R<(C_1+G)(1-\de p^2) $ and recalling that $s$ is also a uniformly strict perfect public strategy, it follows that $r> \frac{3}{4}C_0(1-\de p^2)$ and therefore $\frac{|\hat \alpha-1|+|1-\alpha|}{r}< \frac{8(C_1+G)}{3C_0}.$ Hence, it follows that 
$\frac{\hat r}{r}\leq \frac{8(C_1+G)}{3C_0}+ 1+(1+\frac{4C_5}{3C_0} (T-R)+\frac{4(C_1+G)}{3C_0} R)\frac{T-S}{R}.$ \qed

\subsubsection{$w$ has a uniformly large basin of attraction: proof of Theorem \ref{w-teo}} \label{w teo proof}

Given Theorem \ref{* type ULBA} we only need to show that $w$ is a star-type strategy. We show first that $w$ satisfies the sufficiently responsive condition, i.e., inequality (\ref{eq on b}), to then show that $(w,w)$ is an uniformly strict perfect public equilibrium.

\emph{$w$ satisfies the sufficiently responsive condition:}

Given the behavior of strategy $w$, if an alternative strategy $s$ interacts with $w$ and defects when $w$ cooperates, then $s$ earns  $T$ in that period ($s$  plays $D$ and $w$ plays $C$), and in the next period $s$ gets
either $S$ or $P$ given that $w$ will play $D$. Therefore, if $I_2=\{\tau: u(h{^\tau}_{s,w/h_{t}})=S\},$ $I_3=\{\tau: u(h{^\tau}_{s,w/h_{t}})=T\},$ $I_4=\{\tau: u(h{^\tau}_{s,w/h_{t}})=P\},$ it follows that if $\tau \in I_3$ then $\tau+1\in I_2\cup I_4$ and so
\begin{eqnarray}\label{cond W}
 b_S+b_P & = & \frac{1-p^2\de}{p^2}\sum_{\tau \in I_2\cup I_4}\de^{\tau-t} p^{2(\tau-t)+2}= p^2\de \frac{1-p^2\de}{p^2} \sum_{\tau \in I_2\cup I_4}\de^{\tau-t-1} p^{2(\tau-t)}\\
& \geq &  p^2\de \frac{1-p^2\de}{p^2} \sum_{\tau-1\in I_3}\de^{\tau-t-1}p^{2(\tau-t-1)+2}= p^2\de b_T,
\end{eqnarray}
where $b_R, b_S, b_T, b_P$ are, as for Definition \ref{frequencies}, as follows:
 $$b_P=\frac{1-p^2\de}{p^2}\sum_{\tau: u(h{^\tau}_{s,w/h_{t}})=P}\,p^{2(\tau-t)+2} \de^{\tau-t},\,\,b_S=\frac{1-p^2\de}{p^2}\sum_{\tau: u(h{^\tau}_{s,w/h_{t}})=S}\, p^{2(\tau-t)+2}\de^{\tau-t},$$
$$b_T=\frac{1-p^2\de}{p^2}\sum_{\tau: u(h{^\tau}_{s,w/h_{t}})=T}\, p^{2(\tau-t)+2}\de^{\tau-t}.$$

Recalling that $2R> T+P>T+S$, it follows that $\ga_2=\frac{R-S}{T-R}$ and $\ga_3=\frac{R-P}{T-R}$ are both greater than one. Hence, $b_S\ga_2+b_P\ga_4> b_S+b_P\geq p^2\de b_T,$ and, for $\de$ and $p$ close to one, inequality (\ref{eq on b}) is satisfied.

\emph{The profile $(w,w)$ is an uniformly strict perfect public equilibrium:} \label{w is sgp}

To prove that $(w,w)$ is an uniformly strict perfect public equilibrium, it is enough to show that $ U_{\delta, p}(w,w/h_t)-U_{\delta, p}(s,w/h_t)>(1-p^2 \delta) C_0$ for every history $h_t$ such that $w(h_t)\neq s(h_t)$ with $C_0$ being a positive number.

We start by showing that the difference of payoffs are bounded in 0-tremble histories.

Let's focus first in the case with $w(h_t)=C$ and $s(h_t)=D$. By the fact that $R=b_R R+ b_S R+b_T R+ b_P R$, $b_R=1- b_S-b_T-b_P$ and by inequality (\ref{cond W}), it follows that 
\begin{eqnarray*} 
U_{\delta, p}(h_{w,w/h_t})-U_{\delta, p}(h_{s,w/h_t})&=& b_S(R-S)+b_T(R-T)+b_P(R-P)\\
&\geq & (b_S+b_P)(R-P)+b_T(R-T)\\
&\geq&   \de p^2 b_T (R-P)+ b_T(R-T)\\
&\geq &  b_T[(1+p^2\de)R-(T+P)].
\end{eqnarray*}
Given that $s(h_t)=D$ and $w(h_t)=C,$ strategy $s$ will obtain payoff $T$ at period $t$ and hence
$b_T\geq 1-p^2\de.$
Since $2R-(T+P)>0$, it follows that, for $\de$ and $p$ large,  $[(1+p^2\de)R-(T+P)]> 2C_0$ where $C_0$ is a positive constant such that $C_0=\min\{\frac{P-S}{4}, \frac{2R-(T+P)}{4}\}.$ Therefore, it follows that
$ U_{\delta, p}(h_{w,w/h_t})-U_{\delta, p}(h_{s,w/h_t})> (1-p^2\de)2C_0,$ provided that $\de$ and $p$ are large.

Let's focus now in the case with $w(h_t)=D$ and $s(h_t)=C$. Observe that in this case $b_S\geq 1-p^2\de$ and calculating again the quantities $b_R, b_S, b_T, b_P$ starting from $t+1$, we get that 
$U_{\delta, p}(h_{s,w/h_t})=(1-p^2\de) S+p^2\de[b_RR+b_SS+b_TT+b_PP].$ Therefore, 
writing $p^2\de R=p^2\de[b_R R+ b_S R+b_T R+ b_P R]$ and arguing as before, 
\begin{eqnarray*}
 & & U_{\delta, p}(h_{w,w/h_t})-U_{\delta, p}(h_{s,w/h_t})\\& &=(1-p^2\de)(P-S)+p^2\de[ b_S(R-S)+b_T(R-T)+b_P(R-P)]\\
& &\geq  (1-p^2\de)(P-S)+p^2\de[(b_S+b_P)(R-P)+b_T(R-T)]\\
& & \geq(1-p^2\de)(P-S)+ p^2\de[p^2\de b_T (R-P)+ b_T(R-T)]\\
& &\geq (1-p^2\de)(P-S)+p^2\de b_T [(1+p^2\de)R-(T+p^2\de P)].
\end{eqnarray*}
Since $2R-(T+P)>0,$ it follows that for $\de$ and $p$ large $U_{\delta, p}(h_{w,w/h_t})-U_{\delta, p}(h_{s,w/h_t})> (1-p^2\de)(P-S).$ Given that $C_0=\min\{\frac{P-S}{4}, \frac{2R-(T+P)}{4}\}$, this implies that $ U_{\delta, p}(h_{w,w/h_t})-U_{\delta, p}(h_{s,w/h_t})> (1-p^2\de)2C_0$ for $\de$ and $p$ large.

We now prove that bounded differences in payoffs in 0-tremble histories implies that the difference is also bounded when all histories are considered.

From condition (\ref{p}) on the minimum value of $p$, we have that $p > \sqrt{1-\frac{1}{16}\frac{C_0}{G}(1-\de)^2}$. Hence, $16G(1-p^2)<C_0(1-\de)^2$. Given that we have assumed that $\de>\frac{1}{2}$ and $p>\sqrt{\de}$ by condition (\ref{p}), it follows that $p^2>\frac{1}{2}$ and $\frac{2}{p^2}<16$. Hence, $\frac{2G(1-p^2)}{p^2(1-\de)}<C_0(1-\de)$, and given that $(1-\de)<(1-p^2\de)$ it follows that
\begin{eqnarray} \label{condi}
\frac{2G(1-p^2)}{p^2(1-\de)}<C_0(1-p^2\de).
\end{eqnarray} 
By Lemma \ref{comp1}, $U_{\delta, p}(w,w/h_t)-U_{\delta, p}(s,w/h_t)>U_{\delta, p}(h_{w,w/h_t})-U_{\delta, p}(h_{s,w/h_t})-\frac{2G(1-p^2)}{p^2(1-\de)}.$
Hence, by $U_{\delta, p}(h_{w,w/h_t})-U_{\delta, p}(h_{s,w/h_t})> (1-p^2\de)2C_0$, it follows that
$$U_{\delta, p}(w,w/h_t)-U_{\delta, p}(s,w/h_t)>(1-p^2\de)2C_0-\frac{2G(1-p^2)}{p^2(1-\de)}.$$
Then, by inequality (\ref{condi}), it follows that $U_{\delta, p}(w,w/h_t)-U_{\delta, p}(s,w/h_t)>(1-p^2\de)C_0,$
for $\de$ and $p$ large.

\subsubsection{$Tn$ has a uniformly large basin of attraction: proof of Theorem \ref{Tn-teo}}

Given Theorem \ref{* type ULBA}, we only need to show that $Tn$ is a star-type strategy. For this we need to prove that $Tn$ satisfies the sufficiently responsive condition, i.e., inequality (\ref{eq on b}), and that $(Tn,Tn)$ is an uniformly strict perfect public equilibrium. Proving the former can be done as for $w$. The reason is that, in the period right after the defection when cooperation was expected, $Tn$ will respond to a in the same way as $w$. Hence, the proof used for $w$ applies to $Tn$ as well. It remains to be shown that $(Tn,Tn)$ is an uniformly strict perfect public equilibrium for a sufficiently large $n$.

\emph{The profile $(Tn,Tn)$ is an uniformly strict perfect public equilibrium:} \label{Tn is sgp}

We need to show that $ U_{\delta, p}(Tn,Tn/h_t)-U_{\delta, p}(s,Tn/h_t)>(1-p^2 \delta) C_0$ for every history $h_t$ such that $Tn(h_t)\neq s(h_t)$ with $C_0$ being a positive number.

We start by showing that the difference of payoffs is bounded in 0-tremble histories.

Let's focus first in the case with $Tn(h_t)=C$ and $s(h_t)=D$. In this case, $n$ periods of punishment are triggered and, assuming that $s$ differs from $Tn$ only at $h_t$(by the one step deviation principle), it follows that
\begin{eqnarray*} 
U_{\delta, p}(h_{Tn,Tn/h_t})-U_{\delta, p}(h_{s,Tn/h_t})&=& \frac{1-p^2\de}{p^2}\left[p^2(R-T)+\sum_{i=1}^{n} p^{2(i+1)}\de^i (R-P) \right]
\end{eqnarray*}
Hence, for $p$ and $\de$ large:
\begin{eqnarray} \label{Tn-ine} 
U_{\delta, p}(h_{Tn,Tn/h_t})-U_{\delta, p}(h_{s,Tn/h_t})&=& (1-p^2\de) \frac{(n+1)R-T-nP}{4}.
\end{eqnarray}
Note that given that $R>P$, it is possible to chose $n$ sufficiently large so that $(n+1)R>T+nP$.

Define $C_0$ as a positive constant such that $C_0=\min\{\frac{P-S}{4}, \frac{(n+1)R-T-nP)}{4}\}.$ Therefore, from (\ref{Tn-ine}), it follows that
$ U_{\delta, p}(h_{Tn,Tn/h_t})-U_{\delta, p}(h_{s,Tn/h_t})> (1-p^2\de)2C_0,$ provided that $\de$ and $p$ are large.

We focus now in the case with $Tn(h_t)=D$ and $s(h_t)=C$. The best case for $s$ is now when they are in the first period of punishment. 

By the one step deviation principle we can assume that $s$ differs from $Tn$ only at $h_t$, and it follows that
\begin{eqnarray*} 
U_{\delta, p}(h_{Tn,Tn/h_t})-U_{\delta, p}(h_{s,Tn/h_t})&=& \frac{1-p^2\de}{p^2}\left[p^2(S-P)+p^{2(n+1)}\de^n (R-P) \right]\\
&\geq&(1-p^2\de)(P-S)
\end{eqnarray*}
Given the definition of $C_0$ it follows that $U_{\delta, p}(h_{Tn,Tn/h_t})-U_{\delta, p}(h_{s,Tn/h_t})\geq (1-p^2\de)2C_0$ for $\de$ and $p$ are large.

We prove next, following the same steps as in Section \ref{w teo proof}, that bounded differences in payoffs in 0-tremble histories imply that the differences are also bounded when all histories are considered.

From condition (\ref{p}) on the minimum value of $p$, we have that $p > \sqrt{1-\frac{1}{16}\frac{C_0}{G}(1-\de)^2}$. Hence, $16G(1-p^2)<C_0(1-\de)^2$. Given that we have assumed that $\de>\frac{1}{2}$ and $p>\sqrt{\de}$ by condition \ref{p}, it follows that $p^2>\frac{1}{2}$ and $\frac{2}{p^2}<16$. Hence, $\frac{2G(1-p^2)}{p^2(1-\de)}<C_0(1-\de)$, and given that $(1-\de)<(1-p^2\de)$ it follows that
\begin{eqnarray} \label{condi2}
\frac{2G(1-p^2)}{p^2(1-\de)}<C_0(1-p^2\de).
\end{eqnarray} 
By Lemma \ref{comp1}, $$U_{\delta, p}(Tn,Tn/h_t)-U_{\delta, p}(s,Tn/h_t)>U_{\delta, p}(h_{Tn,Tn/h_t})-U_{\delta, p}(h_{s,Tn/h_t})-\frac{2G(1-p^2)}{p^2(1-\de)}.$$
Hence, by $U_{\delta, p}(h_{Tn,Tn/h_t})-U_{\delta, p}(h_{s,Tn/h_t})> (1-p^2\de)2C_0$, it follows that
$$U_{\delta, p}(Tn,Tn/h_t)-U_{\delta, p}(s,Tn/h_t)>(1-p^2\de)2C_0-\frac{2G(1-p^2)}{p^2(1-\de)}.$$
Then, by inequality (\ref{condi2}), it follows that $U_{\delta, p}(Tn,Tn/h_t)-U_{\delta, p}(s,Tn/h_t)>(1-p^2\de)C_0,$ for $\de$ and $p$ large. \qed

\end{document}